\documentclass[12pt]{article}
\usepackage{amsmath,amsthm,amsfonts,amssymb}
\usepackage[pdftex,pdfborder={0 0 0}]{hyperref}
\usepackage{fullpage}
\usepackage{multirow}
\usepackage{array}
\usepackage{bbm}
\usepackage[noblocks]{authblk}
\usepackage{verbatim}
\usepackage{tikz}
\usepackage{float}
\usepackage{enumerate}
\usepackage{diagbox}
\usepackage{tabu}

\newtheorem{theorem}{Theorem}[section]

\newtheorem{proposition}[theorem]{Proposition}

\theoremstyle{definition}
\newtheorem{definition}[theorem]{Definition}

\newlength{\Oldarrayrulewidth}

\newcommand{\N}{\mathbb{N}}

\newcommand{\Z}{\mathbb{Z}}

\renewcommand{\gcd}{\textup{GCD}}

\renewcommand{\mod}[2]{\equiv#1\textup{ (mod }#2\textup{)}}

\def\m@th{\mathsurround=0pt}
\def\sm#1{\null\,\vcenter{\baselineskip9pt\lineskip.23ex\m@th
    \ialign{\hfil$\scriptstyle##$\hfil&&\ \hfil$\scriptstyle##$\hfil\crcr
    \mathstrut\crcr\noalign{\kern-\baselineskip}
    #1\crcr\mathstrut\crcr\noalign{\kern-\baselineskip}}}\,}
\def\smnp#1{\null\,\vcenter{\baselineskip9pt\lineskip.23ex\m@th
    \ialign{\hfil$\scriptstyle##$\hfil&&\ \ \hfil$\scriptstyle##$\hfil\crcr
    \mathstrut\crcr\noalign{\kern-\baselineskip}
    #1\crcr\mathstrut\crcr\noalign{\kern-\baselineskip}}}\,}

\begin{document}

\title{The Edge-Distinguishing Chromatic Number of Petal Graphs, Chorded Cycles, and Spider Graphs}
\author[1]{Grant Fickes\thanks{gfickes@email.sc.edu}}
\author[2]{Wing Hong Tony Wong\thanks{wong@kutztown.edu}}
\affil[1]{Department of Mathematics, University of South Carolina}
\affil[2]{Department of Mathematics, Kutztown University of Pennsylvania}
\date{\today}

\maketitle

\begin{abstract}
The edge-distinguishing chromatic number (EDCN) of a graph $G$ is the minimum positive integer $k$ such that there exists a vertex coloring $c:V(G)\to\{1,2,\dotsc,k\}$ whose induced edge labels $\{c(u),c(v)\}$ are distinct for all edges $uv$. Previous work has determined the EDCN of paths, cycles, and spider graphs with three legs. In this paper, we determine the EDCN of petal graphs with two petals and a loop, cycles with one chord, and spider graphs with four legs. These are achieved by graph embedding into looped complete graphs.
\end{abstract}


\section{Introduction}\label{sec:intro}

Let the graph $G$ be composed of a simple graph together with at most one loop at each vertex, and let $V(G)$ and $E(G)$ be the vertex set and the edge set of $G$ respectively. Let $[k]=\{1,2,\dotsc,k\}$ denote a set of $k$ colors, and let $\scalebox{0.8}{$\dbinom{[k]}{i}$}$ denote the set of $i$-subsets of $[k]$. For every vertex coloring $c:V(G)\to[k]$, there is an induced edge coloring $c':E(G)\to\scalebox{0.8}{$\dbinom{[k]}{1}$}\cup\scalebox{0.8}{$\dbinom{[k]}{2}$}$, defined by $c'(uv)=\{c(u),c(v)\}$ for each edge $uv$. A vertex coloring $c$ is considered \emph{edge-distinguishing} if $c'$ is injective. The \emph{edge-distinguishing chromatic number (EDCN)} of $G$, denoted by $\lambda(G)$, is the minimum integer $k$ such that an edge-distinguishing vertex coloring with $k$ colors exists for $G$.

The concept of coloring certain elements of a graph $G$ to distinuguish another set of elements associated with $G$ has many variations. For instance, the problem of vertex-distinguishing edge coloring has been studied in the literature \cite{att,bbs,bhlw}, and the problem of distinguishing graph automorphisms through vertex coloring has also been studied \cite{ct,tymoczko}.

This notion of edge-distinguishing vertex coloring, also called a \emph{line-distinguishing vertex coloring}, was first introduced by Frank et al.\ \cite{fhp}. This problem was further studied by Al-Wahabi et al.\ \cite{abhu}, Zagaglia Salvi \cite{salvi}, Fickes and Wong \cite{fw}, etc. Brunton et al.\ \cite{bwg} considered the case when the edge coloring function $c'$ is also surjective. It is worth mentioning that the more popular notion of harmonious chromatic number typically refers to a slight variation of the EDCN \cite{mitchem}. More specifically, if $G$ is a simple graph, then the \emph{harmonious chromatic number} of $G$, denoted by $h(G)$, is the minimum integer $k$ such that an edge-distinguishing proper vertex coloring with $k$ colors exists for $G$. It is obvious that $\lambda(G)\leq h(G)$.

Determining $\lambda(G)$ for a general graph $G$ is NP-complete \cite{hk}. Consequently, most work in the literature focuses on providing bounds on $h(G)$ and thus $\lambda(G)$ \cite{aaeejr} or studying the asymptotic behavior of $h(G)$ \cite{bbw} for various families of graphs $G$. When trying to determine the exact formula for $\lambda(G)$ or $h(G)$, only very limited families have been tackled. For example, $\lambda(G)$ is determined for paths and cycles \cite{abhu,fhp}, and $h(G)$ is determined for complete $r$-ary trees \cite{edwards}.

We are interested in determining the exact formula for the EDCN of other families of graphs; however, pushing the results from paths and cycles to other graphs seems formidable, and it has not been done since the 1980's until very recently. A path is a tree with maximum degree $2$, so a natural extension of a path is a \emph{spider graph}, which has a unique vertex with degree at least $3$, often called the \emph{central vertex}. Each path between the central vertex and a leaf (i.e., a degree $1$ vertex) is called a \emph{leg}. We denote a spider graph as $S_{\ell_1,\ell_2,\dotsc,\ell_\Delta}$, where $\Delta\geq3$ is the number of legs, and $1\leq\ell_1\leq\ell_2\leq\dotsc\leq\ell_\Delta$ represent the number of edges in each leg. Here are some published results on the EDCNs of paths, cycles, and spider graphs in the literature. 

\begin{theorem}[\cite{abhu,fhp}]
Let $G=P_n$ be a path with $n$ vertices. Then
$$\lambda(P_n)=\left\{\begin{array}{ll}
1& \text{if }n\leq2;\text{ and}\\
\min\left\{2\left\lceil\sqrt{\frac{n-2}{2}}\right\rceil,2\left\lceil\frac{1+\sqrt{8n-7}}{4}\right\rceil-1\right\}& \text{if }n\geq3.
\end{array}\right.$$
\end{theorem}

\begin{theorem}[\cite{abhu,fhp}]
Let $G=C_n$ be a cycle with $n$ vertices, where $n\geq3$. Then
$$\lambda(C_n)=\min\left\{2\left\lceil\sqrt{\frac{n}{2}}\right\rceil,2\left\lceil\frac{1+\sqrt{8n+1}}{4}\right\rceil-1\right\}.$$
\end{theorem}

\begin{theorem}[\cite{fw}]\label{Sell1ell2ell3EDCN}
Let $S_{\ell_1,\ell_2,\ell_3}$ be a spider graph with $3$ legs, where $\ell_1+\ell_2+\ell_3=L$. Then
$$\lambda(S_{\ell_1,\ell_2,\ell_3})=\left\{\begin{array}{ll}
3& \text{if $3\leq L\leq 6$ and $\ell_1=1$;}\\
4& \text{if $L=6$ and $\ell_1=2$;}\\
\left\lceil\frac{-1+\sqrt{1+8L}}{2}\right\rceil& \text{if $L\geq7$ and $\left\lceil\frac{-1+\sqrt{1+8L}}{2}\right\rceil$ is odd; and}\\
\left\lceil\sqrt{2L-4}\right\rceil& \text{otherwise.}
\end{array}
\right.$$
\end{theorem}

\begin{theorem}[\cite{fw}]\label{SDeltaellEDCN}
Let $S_{\ell_1,\ell_2,\dotsc,\ell_\Delta}$ be a spider graph with $\Delta$ legs, where $2\leq\ell_i\leq\frac{\Delta+3}{2}$ for each $1\leq i\leq\Delta$. Then
$$\lambda(S_{\ell_1,\ell_2,\dotsc,\ell_\Delta})=\Delta+1.$$
\end{theorem}

Although the edge-distinguishing chromatic number was introduced almost $40$ years ago, the above list seemingly exhausts all major progress on the EDCN of various families of graphs. In this paper, we extend this list by studying spider graphs with four legs. The technique used to obtain Theorem \ref{Sell1ell2ell3EDCN} was too cumbersome to apply, so we adopt a different approach here. We try to ``fold up" or embed a spider graph with four legs into petal graphs and cycles with one chord to aid our discussions. Through this method, we determine the EDCN of three families of graphs: some petal graphs, all cycles with one chord, and all spider graphs with four legs.

To determine the EDCN of various graphs, let us first introduce a simple tool to establish lower bounds.

\begin{proposition}[\cite{fw}]\label{Deltanbdeg>1}
Let $G$ be a simple graph with maximum degree $\Delta(G)$. Then $\lambda(G)\geq\Delta(G)$. Furthermore, if there exists a vertex $u$ of $G$ such that $\deg(u)=\Delta(G)$ and every neighbor of $u$ has degree $>1$, then $\lambda(G)\geq\Delta(G)+1$.
\end{proposition}

Next, we present the main tool to determine the EDCN of a graph, namely Theorem \ref{embedding}. Let $K_k$ denote the complete graph on $k$ vertices $\{v_0,v_1,v_2,\dotsc,v_{k-1}\}$, and let $v_iv_j$ denote the edge between vertices $v_i$ and $v_j$. If we attach a loop $v_iv_i$ at every vertex $v_i$ of $K_k$, then we obtain a new graph, which is denoted by $K_k^*$. In other words, 
$$K_k^*=K_k\cup\{v_0v_0,v_1v_1,v_2v_2,\dotsc,v_{k-1}v_{k-1}\}.$$
A \emph{graph homomorphism} is a function from the vertex set of one graph to the vertex set of another that preserves edges. An \emph{embedding} of a graph $G$ in $K_k^*$ refers to a graph homomorphism from $G$ to $K_k^*$ that induces an injection from $E(G)$ to $E(K_k^*)$.

\begin{theorem}[\cite{fhp}]\label{embedding}
Let $G$ be a simple graph, and let $k$ be a positive integer. Then $\lambda(G)\leq k$ if and only if there is an embedding of $G$ in $K_k^*$.
\end{theorem}

Another trivial but useful observation is the following.

\begin{proposition}\label{subgraph}
Let $H$ be a subgraph of $G$. If $G$ can be embedded in $K_k^*$, then $H$ can also be embedded in $K_k^*$, and hence, $\lambda(H)\leq\lambda(G)$.
\end{proposition}

In Section \ref{sec:kodd}, we begin by proving the necessary and sufficient conditions to embed a subfamily of caterpillar trees in $K_k^*$ when $k$ is odd. Then, we proceed to prove the necessary and sufficient conditions to embed ``petal graphs" with two petals and one loop, cycles with one chord, and finally spider graphs with four legs in $K_k^*$. For each of these families of graphs, the treatment for odd $k$ is different from even $k$, so we separate our discussion on odd and even $k$ into Sections \ref{sec:kodd} and \ref{sec:keven} respectively. Lastly, in Section \ref{sec:edcn}, we determine the EDCN of these petal graphs, chorded cycles, and spider graphs.

\section{$k$ is odd}\label{sec:kodd}

\subsection{Caterpillars}

In the study of graph labelings, caterpillar trees are often the first object studied after paths. Naturally, this is where our discussion begins.

\begin{definition}
Let $\ell$ be a positive integer and $m$ be a nonnegative integer. Let $CP_\ell^{i_1,i_2,\dotsc,i_m}$ denote a \emph{caterpillar tree} with the \emph{central path} $y_0y_1y_2\dotsb y_\ell$ and $m$ extra edges $y_{\ell+1}y_{i_1},\linebreak y_{\ell+2}y_{i_2},\dotsc,y_{\ell+m}y_{i_m}$, where $0<i_1\leq i_2\leq\dotsb\leq i_m<\ell$.
\end{definition}

If $i_1,i_2,\dotsc,i_m$ are distinct and $m\leq k$, then we would like to embed such caterpillar trees in $K_k^*$ with all the extra edges $y_{\ell+1}y_{i_1},y_{\ell+2}y_{i_2},\dotsc,y_{\ell+m}y_{i_m}$ embedded as loops. Hence, the main question is whether we can embed the central path $y_0y_1y_2\dotsb y_\ell$ in $K_k^*$ such that $y_{i_1},y_{i_2},\dotsc,y_{i_m}$ are mapped to distinct vertices. The following theorem by Dvo\v{r}\'{a}k et.\ al.\ answers our question positively with only a few exceptions.

\begin{theorem}[\cite{dhl}]\label{blackcycle}
Let $k\geq3$ be odd, let $\mathcal{C}=z_0z_1z_2\dotsb z_{\binom{k}{2}-1}z_0$ be a cycle of length $\binom{k}{2}$ such that exactly $k$ vertices, namely $z_{i_1},z_{i_2},\dotsc,z_{i_k}$, are black, where $0\leq i_1,i_2,\dotsc,i_k\leq\binom{k}{2}-1$. Then there exists an embedding of $\mathcal{C}$ in $K_k$ that is injective on the black vertices if and only if either $k\neq5$ or $k=5$ and $\{i_1,i_2,i_3,i_4,i_5\}\not\mod{\{0+c,3+c,4+c,6+c,7+c\}\text{ or }\{0+c,1+c,3+c,7+c,9+c\}}{10}$ for every $c\in\Z_{10}$.
\end{theorem}

This result is paramount to the proof of the following theorem.

\begin{theorem}\label{caterpillarembedoddk}
Let $k$ be odd, and let $m\leq k$. Consider caterpillar trees $CP_\ell^{i_1,i_2,\dotsc,i_m}$ such that $0<i_1<i_2<\dotsb<i_m<\ell$ and $CP_\ell^{i_1,i_2,\dotsc,i_m}\neq CP_4^{1,3}$. Such caterpillar trees can be embedded in $K_k^*$ if and only if $\ell+m\leq\binom{k+1}{2}$.
\end{theorem}

\begin{proof}
The number of edges in $CP_\ell^{i_1,i_2,\dotsc,i_m}$ and $K_k^*$ are $\ell+m$ and $\binom{k+1}{2}$ respectively, so $\ell+m\leq\binom{k+1}{2}$ is a necessary condition for $CP_\ell^{i_1,i_2,\dotsc,i_m}$ to be embedded in $K_k^*$.

To prove that the condition is sufficient, if $\ell'+m<\binom{k+1}{2}$, we can first consider $CP_{\ell'}^{i_1,i_2,\dotsc,i_m}$ as a subgraph of $CP_\ell^{i_1,i_2,\dotsc,i_m}$, where $\ell+m=\binom{k+1}{2}$. Due to Proposition \ref{subgraph}, it suffices to show that $CP_\ell^{i_1,i_2,\dotsc,i_m}$ can be embedded in $K_k^*$ if $\ell+m=\binom{k+1}{2}$ and $CP_\ell^{i_1,i_2,\dotsc,i_m}\neq CP_4^{1,3}$.

If $k=1$, then $\ell+m=\binom{k+1}{2}=1$, implying that $\ell=1$ and $m=0$. It is obvious that $CP_1$, which is a path with two vertices, can be embedded in $K_1^*$, which is a loop at one vertex. If $k=3$ or $5$, we exhaust all possible caterpillar trees that satisfy the given conditions, and the only caterpillar tree that cannot be embedded in $K_k^*$ is $CP_4^{1,3}$. If $k\geq7$, we need the following claim so that we can apply Theorem \ref{blackcycle} in our proof.

\underline{\textit{Claim}}: If $k\geq7$ and $m<k$, then there exists a set
$$B=\{y_{j_1},y_{j_1+1},y_{j_2},y_{j_2+1},\dotsc,y_{j_{k-m}},y_{j_{k-m}+1}\}$$
of $2(k-m)$ distinct vertices, where $j_1<j_2<\dotsb<j_{k-m}<\ell$, such that $\{y_{i_1},y_{i_2},\dotsc,y_{i_m}\}\cap B=\emptyset$.

\underline{\textit{Proof of Claim}}: Let
$$J=\big\{j\in\N:j<\ell,\ j\text{ is odd},\text{ and }\{j,j+1\}\cap\{i_1,i_2,\dotsc,i_m\}=\emptyset\big\}.$$
Note that $\left\lfloor\frac{\ell}{2}\right\rfloor\geq\frac{1}{2}(\ell-1)=\frac{1}{2}\left(\binom{k+1}{2}-m-1\right)\geq\frac{1}{2}\left(\binom{k+1}{2}-k\right)=\frac{1}{2}\binom{k}{2}\geq k$ when $k\geq7$. Hence, $|J|\geq\left\lfloor\frac{\ell}{2}\right\rfloor-m\geq k-m$. As a result, our claim follows by picking distinct integers $j_1,j_2,\dotsc,j_{k-m}$ from $J$.

Assume $k\geq7$. Let $CP_\ell^{i_1,i_2,\dotsc,i_m}$ be a caterpillar tree that satisfies the conditions given by the statement of the theorem. Let $j_0=-1$ and $j_{k-m+1}=\ell$, and if $m<k$, then let $B$ be the set of vertices as provided by the claim. Let
$$\phi:\{y_0,y_1,y_2,\dotsc,y_{\ell+m}\}\to\{z_0=z_{\binom{k}{2}},z_1,z_2,\dotsc,z_{\binom{k}{2}-1}\}$$
be defined such that
$$\phi(y_\alpha)=\left\{\begin{tabular}{ll}
$z_{\alpha-\beta}$&if $j_\beta<\alpha\leq j_{\beta+1}$ for some $0\leq\beta\leq k-m$; and\\
$\phi(y_{i_\gamma})$&if $\alpha=\ell+\gamma$ for some $1\leq\gamma\leq m$.
\end{tabular}\right.$$
This function defines an embedding from $CP_\ell^{i_1,i_2,\dotsc,i_m}$ to a graph $\mathcal{C}^*$ that consists of a cycle $z_0z_1z_2\dotsb z_{\binom{k}{2}-1}z_{\binom{k}{2}}$ of length $\binom{k}{2}$ together with the set of $k$ loops
$$\{\phi(y_{i_\gamma})\phi(y_{\ell+\gamma}):1\leq\gamma\leq m\}\cup\{\phi(y_{j_{\beta+1}})\phi(y_{j_{\beta+1}+1})=z_{j_{\beta+1}-\beta}z_{j_{\beta+1}-\beta}:0\leq\beta\leq k-m-1\}.$$
Note that these $k$ loops are at distinct vertices due to the constructions in our claim and our function $\phi$.

Define the cycle $\mathcal{C}$ by removing all the loops from $\mathcal{C}^*$ and coloring the vertices
$$\phi(y_{i_1}),\phi(y_{i_2}),\dotsc,\phi(y_{i_m}),\phi(y_{j_1}),\phi(y_{j_2}),\dotsc,\phi(y_{j_{k-m}})$$
black. Note that these are all the vertices that have loops in $\mathcal{C}^*$. By Theorem \ref{blackcycle}, there exists an embedding of $\mathcal{C}$ in $K_k$ that is injective on the black vertices. As a result, there is an embedding of $\mathcal{C}^*$ in $K_k^*$ with the loops embedded as loops, which completes our proof.
\end{proof}

\subsection{Petal graphs}

A petal graph can be considered as a resultant graph by connecting the legs of a spider graph. It can also be considered as a resultant graph by identifying vertices of a chorded cycle. Hence, it is logical to study the embedding of petal graphs in $K_k^*$ prior to the subsections on cycles with one chord and spider graphs.

\begin{definition}
Let  $m\in\N$ such that $m\geq2$, and let $c_1,c_2,\dotsc,c_m\in\N$ such that $c_1\leq c_2\leq\dotsb\leq c_m$. Let $P_{c_1,c_2,\dotsc,c_m}$ be the \emph{petal graph} that has vertices $u_0,u^1_1,u^1_2,\dotsc,u^1_{c_1-1},\linebreak u^2_1,u^2_2,\dotsc,u^2_{c_2-1},\dotsc,u^m_1,u^m_2,\dotsc,u^m_{c_m-1}$, and for each $i=1,2,\dotsc,m$, the edges form a cycle $u_0u^i_1u^i_2\dotsc u^i_{c_i-1}u_0$ of length $c_i$, which is called the \emph{$i$-th petal} of the petal graph. If $c_i=1$, then the $i$-th petal of $P_{c_1,c_2,\dotsc,c_m}$ is simply a loop.
\end{definition}

We are going to show that petal graphs $P_{1,c_2,c_3}$ can be embedded in $K_k^*$ when $k$ is odd as long as some trivial necessary conditions are satisfied.

\begin{theorem}\label{petalembedoddk}
If $k$ is odd, then $P_{1,c_2,c_3}$ can be embedded in $K_k^*$ if and only if $c_2\geq3$ and $1+c_2+c_3\leq\binom{k+1}{2}$.
\end{theorem}

\begin{proof}
The only if direction is trivial. For the if direction, let $n=c_2+c_3$. If $k\leq3$, then $n\leq\binom{3+1}{2}-1=5$, contradicting that $c_2\geq3$. Hence, it is only meaningful to consider $k\geq5$. In all of the following cases, we embed the first petal $u_0u_0$ of $P_{1,c_2,c_3}$ as the loop $v_0v_0$ in $K_k^*$.

Without loss of generality, we can assume that $n>\binom{k-1}{2}-1$. Otherwise, if $n\leq\binom{k-1}{2}-1$, we can consider embedding $P_{1,c_2,c_3}$ in $K_{k-2}^*$, which can be further embedded in $K_k^*$ trivially. Let $h=\binom{k+1}{2}-1-n$ be the gap between $n$ and $\binom{k+1}{2}-1$. Based on our assumption, $0\leq h\leq\binom{k+1}{2}-1-\binom{k-1}{2}=2k-2$.

 To embed $P_{1,c_2,c_3}$ in $K_k^*$, after embedding the petal $u_0u_0$ as the loop $v_0v_0$ in $K_k^*$, it suffices to partition the edges of $K_k^*-v_0v_0$ into three subgraphs $H_0$, $H_1$, and $H_2$ that satisfy all the following conditions.
\begin{enumerate}
\item The number of edges in $H_0$, $H_1$, and $H_2$ are $h$, $c_2$, and $c_3$ respectively.
\item\label{H1H2connected} After removing isolated vertices, $H_1$ and $H_2$ are connected.
\item\label{H1H2posdeg} The degree of $v_0$ in each of $H_1$ and $H_2$ is positive.
\item\label{H0H1H2evendeg} All degrees in $H_0$, $H_1$, and $H_2$ are even.
\end{enumerate}
This is because the petals $u_0u^2_1u^2_2\dotsb u^2_{c_2-1}u_0$ and $u_0u^3_1u^3_2\dotsb u^3_{c_3-1}u_0$ can then be embedded as $H_1$ and $H_2$ respectively, since $H_1$ and $H_2$ are Eulerian by conditions \ref{H1H2connected} and \ref{H0H1H2evendeg}, and $u_0$ can be embedded as $v_0$ by condition \ref{H1H2posdeg}.

\begin{center}\textbf{Construction of $H_1$}\end{center}

Let $\ell$ be the unique odd integer such that $\binom{\ell-1}{2}-1<c_2\leq\binom{\ell+1}{2}-1$. Note that $\ell\geq3$. Let $g=\binom{\ell+1}{2}-1-c_2$ be the gap between $c_2$ and $\binom{\ell+1}{2}-1$, which satisfies $0\leq g\leq\binom{\ell+1}{2}-1-\binom{\ell-1}{2}=2\ell-2$. Let $S$ be the set of $g$ edges in $K_k^*$ defined as
$$S=\begin{cases}
\{v_1v_1,v_2v_2,\dotsc,v_gv_g\}&\text{if }0\leq g\leq\ell-1;\text{ and}\\
\{v_0v_1v_2\dotsb v_{\ell-1} v_0,v_1v_1,v_2v_2,\dotsc,v_{g-\ell}v_{g-\ell}\}&\text{if }\ell\leq g\leq 2\ell-2.
\end{cases}$$
Let $V_1=\{v_0,v_1,v_2,\dotsc,v_{\ell-1}\}$ and $V_2=\{v_\ell,v_{\ell+1},\dotsc,v_{k-1}\}$.  Let $L_1$ and $L_2$ be the induced subgraphs of $K_k^*$ on $V_1$ and $V_2$ respectively, and let $L_b$ be the complete bipartite graph between $V_1$ and $V_2$.

Define the subgraph $H_1$ as $L_1-v_0v_0-S$. It is clear from the construction that the following conditions hold.
\begin{enumerate}
\item The number of edges in $H_1$ is $c_2$.
\item After removing $V_2$, $H_1$ is connected since
\begin{itemize}
\item if $\ell=3$, then $g=\binom{\ell+1}{2}-1-c_2=5-c_2\leq5-3=2$, so $H_1$ contains the cycle $v_0v_1v_2v_0$; and
\item if $\ell>3$, then $H_1$ contains the cycle $v_0v_2v_4\dotsb v_{\ell-1}v_1v_3\dotsb v_{\ell-2}v_0$.
\end{itemize}
\item The degree of $v_0$ is positive since $H_1$ is connected after $V_2$ is removed.
\item All degrees in $H_1$ are even.
\end{enumerate}

\begin{center}\textbf{Construction of $H_0$}\end{center}

Let $g'$ be the number of loops in $S$. If $h\leq g'$, then let $H_0$ be the subgraph that contains $h$ loops in $S$. Otherwise, we have the following claim.

\underline{\textit{Claim $1$}}: If $h>g'$, then $k>\ell$.

\underline{\textit{Proof of Claim $1$}}: If $k=\ell$, then $\binom{k-1}{2}\leq c_2\leq c_3$, thus $0\leq h=\binom{k+1}{2}-1-n\leq\binom{k+1}{2}-1-2c_2\leq\binom{k+1}{2}-1-2\binom{k-1}{2}=\frac{-(k-6)(k-1)}{2}$, which implies $k=5$ and $h\leq2$. Note that $\binom{5-1}{2}\leq c_2\leq\frac{1}{2}\left(\binom{5+1}{2}-1\right)$, so $c_2=6$ or $7$. Hence, $g=\binom{5+1}{2}-1-c_2=8$ or $7$, meaning that $S$ has $8$ or $7$ edges. As a result, the number of loops in $S$, denoted by $g'$, is $3$ or $2$. However, $h\leq2$, contradicting the condition that $h>g'$.

Under the assumption that $h>g'$, we have $k-\ell\geq2$. Index the vertices in $V_1-\{v_0\}$ in pairs as $v_{2s-1}$ and $v_{2s}$, where $1\leq s\leq\frac{\ell-1}{2}$, and index the vertices in $V_2$ in pairs as $v_{\ell+2t-2}$ and $v_{\ell+2t-1}$, where $1\leq t\leq\frac{k-\ell}{2}$. Since $\ell\geq3$, the number of edges in the complete bipartite graph between $V_1-\{v_0\}$ and $V_2$ is $(\ell-1)(k-\ell)\geq2(k-3)=2k-6$. Let $h'=h-g'$. We are going to define $H_0$ based on the following cases.

Case $1$: $k-\ell\geq4$. Let $\widetilde{h}$ be the smallest positive integer such that $h'\mod{\widetilde{h}}{4}$. Note that $h'-\widetilde{h}<h'\leq h\leq2k-2$. Together with the fact that $h'-\widetilde{h}\equiv2k-2\mod{0}{4}$, we have $h'-\widetilde{h}\leq2k-6$.

Define the subgraph $H_0$ as the disjoint union of 
\begin{itemize}
\item $\frac{h'-\widetilde{h}}{4}$ copies of $4$-cycles $v_{2s-1}v_{\ell+2t-2}v_{2s}v_{\ell+2t-1}v_{2s-1}$,
\item $g'$ loops in $S$, and
\item $\widetilde{h}$ loops in $L_2$.
\end{itemize}
The number of edges in $H_0$ is $\frac{h'-\widetilde{h}}{4}\cdot4+g'+\widetilde{h}=h'+g'=h$, and it is clear that all degrees in $H_0$ are even.

Case $2$: $k-\ell=2$, i.e., $\ell=k-2$. The number of edges in $K_k^*-v_0v_0-H_0$ is at least $2c_2$, so $h'\leq h\leq\binom{k+1}{2}-1-2c_2\leq\binom{k+1}{2}-1-2\binom{\ell-1}{2}=\binom{k+1}{2}-1-2\binom{k-3}{2}=\frac{1}{2}(-k^2+15k-26)$, in addition to the usual bound $h'\leq h\leq2k-2$. In order for $\frac{1}{2}(-k^2+15k-26)$ to be nonnegative, we must have $k\leq13$. Here is a list of values of $k\leq13$ and the corresponding upper bounds on $h$.
\begin{center}
\begin{tabular}{|c|c|c|c|c|c|}
\hline
$k$& 5& 7& 9& 11& 13\\
\hline
$h\leq$& 8& 12& 14& 9& 0\\
\hline
\end{tabular}
\end{center}

\underline{\textit{Claim $2$}}: If $k\leq13$, then $h'\leq2k-4$.

\underline{\textit{Proof of Claim $2$}}: Since $h'\leq h$, the only possible cases for $h'>2k-4$ are when $(k,h)=(5,7)$, $(5,8)$, $(7,11)$, and $(7,12)$. Note that $\max\left\{3,\binom{k-3}{2}\right\}=\max\left\{3,\binom{\ell-1}{2}\right\}\leq c_2\leq\frac{1}{2}\left(\binom{k+1}{2}-1-h\right)$. When $k=5$ and $h=7$ or $8$, we have $c_2=3$, so $H_1$ is the cycle $v_0v_1v_2v_0$, and $g'=2$. Hence, $h'=h-g'=5$ or $6$. When $k=7$, we have 
$$c_2=\begin{cases}
6,\ 7,\text{ or }8&\text{if }h=11;\text{ and}\\
6\text{ or }7&\text{if }h=12.
\end{cases}$$
As a result, $H_1$ is the union of the cycle $v_0v_2v_4v_1v_3v_0$ and $c_2-5$ loops in $L_1-v_0v_0$. Hence, 
$$g'=\begin{cases}
3,\ 2,\text{ or }1&\text{if }h=11;\text{ and}\\
3\text{ or }2&\text{if }h=12.
\end{cases}$$
This implies 
$$h'=h-g'=\begin{cases}
8,\ 9,\text{ or }10&\text{if }h=11;\text{ and}\\
9\text{ or }10&\text{if }h=12.
\end{cases}$$
In all cases, $h'\leq2k-4$. This completes the proof of Claim~$2$.

Next, let $\widetilde{h}$ be the smallest nonnegative integer such that $h'\mod{\widetilde{h}}{4}$, which is slightly different from Case 1. Note that Claim~$2$, together with the fact that $h'-\widetilde{h}\equiv2k-2\mod{0}{4}$, implies $h'-\widetilde{h}\leq2k-6$. If $\widetilde{h}\leq 2$, define the subgraph $H_0$ as the union of
\begin{itemize}
\item $4$-cycles $v_{2s-1}v_{k-2}v_{2s}v_{k-1}v_{2s-1}$, where $1\leq s\leq\frac{h'-\widetilde{h}}{4}$,
\item $g'$ loops in $S$, and
\item $\widetilde{h}$ loops in $L_2$.
\end{itemize}
If $\widetilde{h}=3$, define the subgraph $H_0$ as the union of
\begin{itemize}
\item $4$-cycles $v_{2s-1}v_{k-2}v_{2s}v_{k-1}v_{2s-1}$, where $1\leq s\leq\frac{h'-\widetilde{h}}{4}$,
\item $g'$ loops in $S$, and
\item the triangle $v_{k-3}v_{k-2}v_{k-1}v_{k-3}$.
\end{itemize}
The number of edges in $H_0$ is $4\cdot\frac{h'-\widetilde{h}}{4}+g'+\widetilde{h}=h'+g'=h$, and it is clear that all degrees in $H_0$ are even.

\begin{center}\textbf{Construction of $H_2$}\end{center}

Define the subgraph $H_2$ as $K_k^*-v_0v_0-H_0-H_1$. From the construction, the following conditions hold.
\begin{enumerate}
\item The number of edges in $H_2$ is $\binom{k+1}{2}-1-h-c_2=n-c_2=c_3$.
\item\label{item:H2connected} After removing isolated vertices in $V_1$, $H_2$ is connected due to the following analysis.
\begin{itemize}
\item If $h\leq g'$ and $k=\ell$, then as provided in the proof of Claim $1$, $k=5$ and $S$ contains the cycle $v_0v_1v_2v_3v_4v_0$, which is completely contained in $H_2$ since $H_0$ does not contain any edges in this cycle.
\item If $h\leq g'$ and $k>\ell$, then $H_2$ contains both $L_2$ and $L_b$.
\item If $h>g'$, then by Claim $1$, $k>\ell$. Hence, in $H_2$, $v_0$ is connected with all vertices in $V_2$, and vertices $v_1,v_2,\dotsc,v_{\ell-1}$ are either isolated vertices and thus removed, or are connected to $v_0$ through the cycle $v_0v_1v_2\dotsb v_{\ell-1}v_0$. 
\end{itemize}
\item The degree of $v_0$ in $H_2$ is positive due to item \ref{item:H2connected}.
\item All degrees in $H_2$ are even since all degrees in $K_k^*$, $v_0v_0$, $H_0$, and $H_1$ are even.
\end{enumerate}
\end{proof}

\subsection{Cycles with one chord}

Caterpillar trees are natural extensions of paths, while chorded cycles are natural extensions of cycles. In this section, we are going to study cycles with one extra chord.

\begin{definition}\label{chordedcycle}
Let $C_n^{\{0,j\}}$ be a cycle on $n$ vertices $w_0,w_1,w_2,\dotsc,w_{n-1}$ with a chord between $w_0$ and $w_j$. Without loss of generality, let $2\leq j\leq\frac{n}{2}$, which also implies $n\geq4$. Note that $j\neq1$, as $w_0w_1$ is already in the cycle $C_n$.
\end{definition}

Note that $C_n^{\{0,2\}}$ contains a triangle $w_0w_1w_2w_0$. When it is embedded in $K_k^*$, this triangle creates an obstruction if $n$ is very close to $\binom{k+1}{2}$. This is reflected by the extra complication in the statement of the following theorem.

\begin{theorem}\label{OddChordedCycle}
Let $k$ be odd. Then $C_n^{\{0,2\}}$ can be embedded in $K_k^*$ if and only if $n\leq\binom{k+1}{2}-3$. When $j\geq3$, then $C_n^{\{0,j\}}$ can be embedded in $K_k^*$ if and only if $n\leq\binom{k+1}{2}-1$. 
\end{theorem}

\begin{proof}
If $k\leq3$, then it is obvious that $C_n^{\{0,j\}}$ cannot be embedded in $K_k^*$. Also, $k\leq3$ is ruled out implicitly by $4\leq n\leq\binom{k+1}{2}-3$ as well as $6\leq2j\leq n\leq\binom{k+1}{2}-1$. Hence, it is only meaningful to consider $k\geq5$.

If $C_n^{\{0,2\}}$ can be embedded in $K_k^*$, then note that the vertices in the triangle $w_0w_1w_2w_0$ must be embedded as distinct vertices in $K_k^*$. Otherwise, we will create double edges or multiple loops at the same vertex, which do not exist in $K_k^*$. Without loss of generality, assume that $w_0$, $w_1$, $w_2$ are embedded as $v_0$, $v_1$, $v_2$ respectively. The degrees of $w_0$ and $w_2$ in the path $w_2w_3\dotsb w_{n-1}w_0$ are odd, while the degrees of $v_0$ and $v_2$ in $K_k^*-v_0v_1v_2v_0$ are even. To create odd degree vertices at $v_0$ and $v_2$ in $K_k^*-v_0v_1v_2v_0$, if we were to forgo only one edge from $K_k^*-v_0v_1v_2v_0$, that edge must be $v_0v_2$. However, $v_0v_2$ does not exist in $K_k^*-v_0v_1v_2v_0$, so at least two edges are forgone from $K_k^*-v_0v_1v_2v_0$ when the path $w_2w_3\dotsb w_{n-1}w_0$ is embedded in $K_k^*-v_0v_1v_2v_0$. Therefore, $n+1\leq\binom{k+1}{2}-2$, or $n\leq\binom{k+1}{2}-3$.

If $n\leq\binom{k+1}{2}-3$, without loss of generality, we can assume that $n>\binom{k-1}{2}-3$. Otherwise, if $n\leq\binom{k-1}{2}-3$, we can consider embedding $C_n^{\{0,2\}}$ in $K_{k-2}^*$, which can be further embedded in $K_k^*$ trivially. Let $h=\binom{k+1}{2}-3-n$ be the gap between $n$ and $\binom{k+1}{2}-3$. Based on our assumption, $0\leq h\leq\binom{k+1}{2}-3-\left(\binom{k-1}{2}-2\right)=2k-2$. Let $H$ be defined such that 
$$H=\begin{cases}
\{v_0v_3v_2,v_0v_0,v_1v_1,v_2v_2,\dotsc,v_{h-1}v_{h-1}\}&\text{if }0\leq h\leq k;\\
\{v_0v_4v_1v_3v_4v_4v_2,v_0v_0,v_1v_1,\dotsc,v_{h-5}v_{h-5}\}&\text{if }k=5\text{ and }6\leq h\leq8;\text{ and}\\
\{v_0v_3v_2,v_{k-1}v_1v_3v_4v_5\dotsb v_{k-1},\\
\hspace{30pt}v_0v_0,v_1v_1,v_2v_2,\dotsc,v_{h-k+1}v_{h-k+1}\}&\text{if }k\geq7\text{ and }k+1\leq h\leq2k-2.
\end{cases}$$
Note that the number of edges in $H$ is $h+2$ in all cases. Now, the chorded cycle $C_n^{\{0,2\}}$ can be embedded in $K_k^*$ such that $w_0$, $w_1$, $w_2$ are embedded as $v_0$, $v_1$, and $v_2$ respectively, and the path $w_2w_3\dotsb w_{n-1}w_0$ is embedded as the Eulerian path in $K_k^*-v_0v_1v_2v_0-H$. This is because the only odd degree vertices in the graph $K_k^*-v_0v_1v_2v_0-H$ are $v_0$ and $v_2$, and this graph is connected after isolated vertices are removed.

When $j\geq3$, the only if direction is trivial. For the if direction, the chorded cycle $C_n^{\{0,j\}}$ can be embedded in the petal graph $P_{1,j,n-j}$ as follows: the chord $w_0w_j$ is embedded as the first petal $u_0u_0$, the path $w_0w_1w_2\dotsb w_j$ is embedded as the second petal $u_0u^2_1u^2_2\dotsb u^2_{j-1}u_0$, and the path $w_jw_{j+1}\dotsb w_{n-1}w_0$ is embedded as the third petal $u_0u^3_1u^3_2\dotsb u^3_{n-j-1}u_0$. Theorem \ref{petalembedoddk} completes the proof.
\end{proof}

\subsection{Spider graphs with three or four legs}

We have provided the definition of spider graphs in Section \ref{sec:intro}, but it is more convenient for our future discussions if we formalize notation.

\begin{definition}\label{defspider}
Let $\Delta\in\N$ such that $\Delta\geq3$, and let $\ell_1,\ell_2,\dotsc,\ell_\Delta\in\N$ such that $\ell_1\leq\ell_2\leq\dotsc\leq\ell_\Delta$. Let $S_{\ell_1,\ell_2,\dotsc,\ell_\Delta}$ be the \emph{spider graph} with vertices $x_0,x^1_1,x^1_2,\dotsc,x^1_{\ell_1},x^2_1,x^2_2,\dotsc,x^2_{\ell_2},\linebreak \dotsc,x^\Delta_1,x^\Delta_2,\dotsc,x^\Delta_{\ell_\Delta}$, and for each $i=1,2,\dotsc,\Delta$, the edges form a path $x_0x^i_1,x^i_2,\dotsc,x^i_{\ell_i}$ between the \emph{central vertex} $x_0$ and the leaf $x^i_{\ell_i}$, which is called the \emph{$i$-th leg} of the spider graph.
\end{definition}

Although the necessary and sufficient conditions for embedding spider graphs with three legs in $K_k^*$ were proved \cite{fw}, we are going to provide a much simplified proof when $k$ is odd, utilizing Theorem \ref{petalembedoddk}.

\begin{theorem}
Let $n=\ell_1+\ell_2+\ell_3$, and let $k$ be odd. If $n\geq7$, then $S_{\ell_1,\ell_2,\ell_3}$ can be embedded in $K_k^*$ if and only if $n\leq\binom{k+1}{2}$.
\end{theorem}

\begin{proof}
If $k\leq3$, then $n\leq\binom{3+1}{2}=6$, violating the condition that $n\geq7$. Hence, it is only meaningful to consider $k\geq5$. Furthermore, $n\geq7$ implies $\ell_3\geq\left\lceil\frac{1}{3}n\right\rceil\geq3$.

The only if direction is trivial. For the if direction, we consider the following cases.

Case $1$: $\ell_1=1$. Hence, the spider graph $S_{1,\ell_2,\ell_3}$ can be embedded in the petal graph $P_{1,3,n-4}$ as follows: the vertices $x_0$ and $x^1_1$ of the first leg are both embedded as $u_0$; the vertices $x^2_1,x^2_2,\dotsc,x^2_{\ell_2}$ of the second leg are embedded as $u^3_1,u^3_2,\dotsc,u^3_{\ell_2}$; and the vertices $x^3_1,x^3_2,\dotsc,x^3_{\ell_3}$ of the third leg are embedded as $u^2_1,u^2_2,u_0,u^3_{n-5},u^3_{n-6},\dotsc,u^3_{\ell_2}$. Theorem \ref{petalembedoddk} completes the proof.

Case $2$: $\ell_1\geq2$. Let $(i,j)=(2,3)$ if $\ell_1+\ell_2-1\leq\ell_3$, and let $(i,j)=(3,2)$ otherwise. The spider graph $S_{\ell_1,\ell_2,\ell_3}$ can be embedded in the petal graph $P_{1,c_2,c_3}$, where $c_i=\ell_1+\ell_2-1$ and $c_j=\ell_3$, as follows: the vertices $x_0,x^1_1,x^1_2,\dotsc,x^1_{\ell_1}$ of the first leg are embedded as $u_0,u_0,u^i_1,u^i_2,\dotsc,u^i_{\ell_1-1}$; the vertices $x^2_1,x^2_2,\dotsc,x^2_{\ell_2}$ of the second leg are embedded as $u^i_{c_i-1},u^i_{c_i-2},\dotsc,u^i_{\ell_1-1}$; and the vertices $x^3_1,x^3_2,\dotsc,x^3_{\ell_3}$ of the third leg are embedded as $u^j_1,u^j_2,\dotsc,u^j_{c_j-1},u_0$. Again, Theorem \ref{petalembedoddk} completes the proof.
\end{proof}

The last theorem of this section provides the necessary and sufficient conditions for embedding spider graphs with four legs, the family of graphs that originally drew our interest, in $K_k^*$ when $k$ is odd.

\begin{theorem}\label{spiderembedoddk}
Let $n=\ell_1+\ell_2+\ell_3+\ell_4$, and let $k$ be odd. If $n\geq7$, then $S_{\ell_1,\ell_2,\ell_3,\ell_4}$ can be embedded in $K_k^*$ if and only if
\begin{enumerate}[$(a)$]
\item $\ell_3=1$ and $n\leq\binom{k+1}{2}-1$, or
\item $\ell_3\geq2$ and $n\leq\binom{k+1}{2}$.
\end{enumerate}
\end{theorem}

\begin{proof}
If $k\leq3$, then $n\leq\binom{3+1}{2}=6$, violating that $n\geq7$. Hence, it is only meaningful to consider $k\geq5$.

Assume $S_{\ell_1,\ell_2,\ell_3,\ell_4}$ can be embedded in $K_k^*$. It is trivial that $n\leq\binom{k+1}{2}$. If $n=\binom{k+1}{2}$, then since $x^1_{\ell_1}$, $x^2_{\ell_2}$, $x^3_{\ell_3}$, and $x^4_{\ell_4}$ are of odd degree in $S_{\ell_1,\ell_2,\ell_3,\ell_4}$ and all vertices in $K_k^*$ are of even degree, at least two of $x^1_{\ell_1}$, $x^2_{\ell_2}$, and $x^3_{\ell_3}$ must be embedded as the same vertex. If $\ell_3=1$, then this is impossible since there are no double edges in $K_k^*$, and hence $n\leq\binom{k+1}{2}-1$.

For the if direction, due to Proposition \ref{subgraph}, it suffices to show that $S_{\ell_1,\ell_2,\ell_3,\ell_4}$ can be embedded in $K_k^*$ if $n=\binom{k+1}{2}-1$ when $\ell_3=1$ and $n=\binom{k+1}{2}$ when $\ell_3\geq2$.

Case $1$: $\ell_3=1$. The spider graph $S_{1,1,1,\ell_4}$ can be embedded in $K_k^*$ as follows: the vertices $x_0$, $x^1_1$, $x^2_1$, and $x^3_1$ are embedded as $v_0$, $v_0$, $v_1$, and $v_2$ respectively; the remaining graph $K_k^*-v_0v_0-v_0v_1v_2v_0$ is an Eulerian graph, so the fourth leg can be embedded as an Eulerian cycle.

Case $2$: $\ell_3\geq2$. Let $(i,j)=(2,3)$ if $\ell_1+\ell_3\leq\ell_2+\ell_4-1$, and let $(i,j)=(3,2)$ otherwise. Define $c_i=\ell_1+\ell_3$ and $c_j=\ell_2+\ell_4-1$. Note that $c_i\geq1+2=3$ and $c_j\geq\max\{n-\ell_1-\ell_3-1,\ell_1+\ell_3-1\}\geq3$ since $n\geq7$. The spider graph $S_{\ell_1,\ell_2,\ell_3,\ell_4}$ can be embedded in the petal graph $P_{1,c_2,c_3}$, where $c_i=\ell_1+\ell_3$ and $c_j=\ell_2+\ell_4-1$ as follows: the vertices $x_0,x^1_1,x^1_2,\dotsc,x^1_{\ell_1}$ of the first leg are embedded as $u_0,u^i_1,u^i_2,\dotsc,u^i_{\ell_1}$; the vertices $x^2_1,x^2_2,\dotsc,x^2_{\ell_2}$ of the second leg are embedded as $u_0,u^j_1,u^j_2,\dotsc,u^j_{\ell_2-1}$; the vertices $x^3_1,x^3_2,\dotsc,x^3_{\ell_3}$ of the third leg are embedded as $u^i_{c_i-1},u^i_{c_i-2},\dotsc,u^i_{c_i-\ell_3}$; and the vertices $x^4_1,x^4_2,\dotsc,x^4_{\ell_4}$ of the fourth leg are embedded as $u^j_{c_j-1},u^j_{c_j-2},\dotsc,u^j_{c_j-\ell_4}$. Theorem \ref{petalembedoddk} completes the proof.
\end{proof}

\section{$k$ is even}\label{sec:keven}

In Section \ref{sec:kodd}, we restrict our consideration to $k$ being odd. In this section, except for caterpillar trees, we prove the results parallel to the theorems in the previous section for even $k$.

To aid our discussions in this section, here are some notations related to an edge decomposition of $K_k^*$. For each $j=0,1,2,\dotsc,\frac{k-2}{2}$, let $D_j$ be the subgraph of $K_k^*$ such that the edge set of $D_j$ is
$$\{v_pv_q:p\leq q,\text{ either }q-p=j\text{ or }q-p=k-j\}.$$
Note that each $D_j$ is a regular degree $2$ graph and has exactly $k$ edges. Furthermore, let $I$ be the perfect matching subgraph of $K_k^*$ such that the edge set of $I$ is
$$\left\{v_iv_{\frac{k}{2}+i}:i=0,1,2,\dotsc,\frac{k-2}{2}\right\}.$$
Figure \ref{figK6} illustrates how $K_6^*$ is decomposed into $D_0$, $D_1$, $D_2$, and $I$.
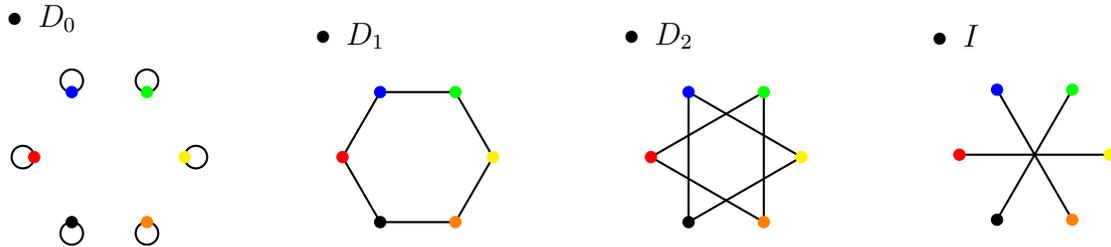
\begin{figure}[H]
\centering
\begin{tikzpicture}[scale=0.55]
{\draw[thick, black] (0,0) -- (2,0) -- (3,1.73) -- (2, 3.46) -- (0, 3.46) -- (-1, 1.73) -- (0,0);}
{\draw[thick, black] (0,0) -- (3,1.73) -- (0,3.46) -- (0,0);}
{\draw[thick, black] (2,0) -- (2,3.46) -- (-1,1.73) -- (2,0);}
{\draw[thick, black] (0,0) -- (2,3.46);}
{\draw[thick, black] (2,0) -- (0,3.46);}
{\draw[thick, black] (3,1.73) -- (-1,1.73);}
{\draw[thick, black](0,-0.3)circle(0.3);}
{\draw[thick, black](2,0-0.3)circle(0.3);}
{\draw[thick, black](3.3,1.73)circle(0.3);}
{\draw[thick, black](2,3.76)circle(0.3);}
{\draw[thick, black](0,3.76)circle(0.3);}
{\draw[thick, black](-1.3,1.73)circle(0.3);}
{\filldraw[black](0,0)circle(0.15);}
{\filldraw[orange](2,0)circle(0.15);}
{\filldraw[yellow](3,1.73)circle(0.15);}
{\filldraw[green](2,3.46)circle(0.15);}
{\filldraw[blue](0,3.46)circle(0.15);}
{\filldraw[red](-1,1.73)circle(0.15);}
{\node at (-2,1.73){$v_0$};}
{\node at (0,4.36){$v_5$};}
{\node at (2,4.36){$v_4$};}
{\node at (4,1.73){$v_3$};}
{\node at (2,-0.9){$v_2$};}
{\node at (0,-0.9){$v_1$};}
\end{tikzpicture}
\caption{$K_6^*$ contains $D_0$, $D_1$, and $D_2$ as shown below}\label{figK6}
\end{figure}
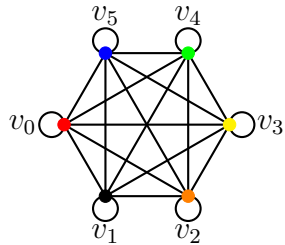
\begin{figure}[H]
\centering
\begin{minipage}{0.24\linewidth}
\centering
\begin{itemize}
\item $D_0$
\end{itemize}
\begin{tikzpicture}[scale=0.5]
{\draw[thick, black](0,-0.3)circle(0.3);}
{\draw[thick, black](2,0-0.3)circle(0.3);}
{\draw[thick, black](3.3,1.73)circle(0.3);}
{\draw[thick, black](2,3.76)circle(0.3);}
{\draw[thick, black](0,3.76)circle(0.3);}
{\draw[thick, black](-1.3,1.73)circle(0.3);}
{\filldraw[black](0,0)circle(0.15);}
{\filldraw[orange](2,0)circle(0.15);}
{\filldraw[yellow](3,1.73)circle(0.15);}
{\filldraw[green](2,3.46)circle(0.15);}
{\filldraw[blue](0,3.46)circle(0.15);}
{\filldraw[red](-1,1.73)circle(0.15);}
\end{tikzpicture}
\end{minipage}
\begin{minipage}{0.24\linewidth}
\centering
\begin{itemize}
\item $D_1$
\end{itemize}
\begin{tikzpicture}[scale=0.5]
{\draw[thick, black] (0,0) -- (2,0) -- (3,1.73) -- (2, 3.46) -- (0, 3.46) -- (-1, 1.73) -- (0,0);}
{\filldraw[black](0,0)circle(0.15);}
{\filldraw[orange](2,0)circle(0.15);}
{\filldraw[yellow](3,1.73)circle(0.15);}
{\filldraw[green](2,3.46)circle(0.15);}
{\filldraw[blue](0,3.46)circle(0.15);}
{\filldraw[red](-1,1.73)circle(0.15);}
\end{tikzpicture}
\end{minipage}
\begin{minipage}{0.24\linewidth}
\centering
\begin{itemize}
\item $D_2$
\end{itemize}
\begin{tikzpicture}[scale=0.5]
{\draw[thick, black] (0,0) -- (3,1.73) -- (0,3.46) -- (0,0);}
{\draw[thick, black] (2,0) -- (2,3.46) -- (-1,1.73) -- (2,0);}
{\filldraw[black](0,0)circle(0.15);}
{\filldraw[orange](2,0)circle(0.15);}
{\filldraw[yellow](3,1.73)circle(0.15);}
{\filldraw[green](2,3.46)circle(0.15);}
{\filldraw[blue](0,3.46)circle(0.15);}
{\filldraw[red](-1,1.73)circle(0.15);}
\end{tikzpicture}
\end{minipage}
\begin{minipage}{0.24\linewidth}
\centering
\begin{itemize}
\item $I$
\end{itemize}
\begin{tikzpicture}[scale=0.5]
\draw[thick, black] (0,0) -- (2,3.46);
\draw[thick, black] (3,1.73) -- (-1,1.73);
\draw[thick, black] (0,3.46) -- (2,0);
{\filldraw[black](0,0)circle(0.15);}
{\filldraw[orange](2,0)circle(0.15);}
{\filldraw[yellow](3,1.73)circle(0.15);}
{\filldraw[green](2,3.46)circle(0.15);}
{\filldraw[blue](0,3.46)circle(0.15);}
{\filldraw[red](-1,1.73)circle(0.15);}
\end{tikzpicture}
\end{minipage}
\end{figure}

\subsection{Petal graphs}

Similar to the case when $k$ is odd, we are going to show that petal graphs $P_{1,c_2,c_3}$ can be embedded in $K_k^*$ when $k$ is even as long as some trivial necessary conditions are satisfied.

\begin{theorem}\label{petalembedevenk}
If $k$ is even, then $P_{1,c_2,c_3}$ can be embedded in $K_k^*$ if and only if $k\geq6$, $c_2\geq3$, and $1+c_2+c_3\leq\frac{k^2}{2}$.
\end{theorem}

\begin{proof}
For the only if direction, the condition $c_2\geq3$ is trivial. If $k\leq4$, then the maximum degree in $K_k^*-I$ is at most $4$, which is less than the maximum degree in $P_{1,c_2,c_3}$. Hence, we must have $k\geq6$. Also, the petal graph $P_{1,c_2,c_3}$ has only even degrees, so when it is embedded in $K_k^*$, the image must also have only even degrees. Since all the $k$ vertices in $K_k^*$ have odd degrees, at least one edge will be missing from each vertex after embedding. In other words, there are at least $\frac{k}{2}$ edges missing, implying that the number of edges in $P_{1,c_2,c_3}$ is at most $\binom{k+1}{2}-\frac{k}{2}=\frac{k^2}{2}$.

For the if direction, let $n=c_2+c_3$. We are going to show that if $k\geq6$, $c_2\geq3$, and $1+n\leq\frac{k^2}{2}$, then $P_{1,c_2,c_3}$ can be embedded in $K_k^*-I$. In all of the following cases, we embed the petal $u_0u_0$ as the loop $v_0v_0$ in $K_k^*$.

Without loss of generality, we can assume that $n>\binom{k}{2}-1$. Otherwise, if $n\leq\binom{k}{2}-1$, we can embed $P_{1,c_2,c_3}$ in $K_{k-1}^*$ by Theorem \ref{petalembedoddk}, which can be further embedded in $K_k^*$ trivially. Let $h=\frac{k^2}{2}-1-n$ be the gap between $n$ and $\frac{k^2}{2}-1$. Based on our assumption, $0\leq h\leq\frac{k^2}{2}-1-\binom{k}{2}=\frac{k}{2}-1$.

Case $1$: $3\leq c_2<k$. Let $H_0$, $H_1$, and $H_2$ be subgraphs of $K_k^*-I$ such that $H_0=\{v_2v_2,v_3v_3,\dotsc,v_{h+1}v_{h+1}\}$, 
$$H_1=\begin{cases}
v_0v_1v_2v_3\dotsb v_{c_2-1}v_0&\text{if }c_2\neq\frac{k}{2}+1;\text{ and}\\
v_0v_1v_1v_2v_3\dotsb v_{\frac{k}{2}-1}v_0&\text{if }c_2=\frac{k}{2}+1,
\end{cases}$$
and $H_2=K_k^*-I-v_0v_0-H_0-H_1$. Note that $H_0$ is well-defined since $h\leq\frac{k}{2}-1\leq k-2$. With a simple count, we see that $H_0$ has $h$ edges, $H_1$ has $c_2$ edges, and $H_2$ has $\binom{k+1}{2}-\frac{k}{2}-1-h-c_2=n-c_2=c_3$ edges. It is clear that all vertices in $H_1$ and $H_2$ are of even degree. Furthermore, $H_2$ is connected since $H_2$ contains $D_2\cup\{v_{k-1}v_0\}-\{v_{c_2-1}v_0\}$. Therefore, $H_1$ and $H_2$ are both Eulerian, so we can embed the petals $u_0u^2_1u^2_2\dotsb u^2_{c_2-1}u_0$ as $H_1$ and $u_0u^3_1u^3_2\dotsb u^3_{c_3-1}u_0$ as $H_2$.

Case $2$: $c_2\geq k$ and $k=6$. Note that $c_2<\frac{1}{2}(1+n)\leq\frac{1}{2}\cdot\frac{k^2}{2}=9$, i.e., $c_2=6$, $7$, or $8$. We can embed the petal $u_0u^2_1u^2_2\dotsb u^2_{c_2-1}u_0$ as $v_0v_1v_3v_5v_4v_2v_0$ along with $c_2-6$ additional loops, and embed the petal $u_0u^3_1u^3_2\dotsb u^3_{c_3-1}u_0$ as $v_0v_5v_1v_2v_3v_4v_0$ along with $c_3-6$ additional loops. Since $c_2-6+c_3-6=n-12\leq5$, we can ensure that the two petals are embedded using different loops in $K_k^*-I-v_0v_0$.

Case $3$: $c_2\geq k$ and $k\geq8$. Let $z$ be the odd integer in $\{\frac{k-2}{2},\frac{k-4}{2}\}$. Note that $\gcd(k,z)=1$. Since $k\geq8$, $z\neq1$. In other words, $D_1$ and $D_z$ are two distinct Hamiltonian cycles in $K_k^*$. Let $c_2'$ and $c_3'$ be the smallest nonnegative integers such that $c_2\mod{c_2'}{k}$ and $c_3\mod{c_3'}{k}$.

Case $3.1$: $c_2'+c_3'<k$. We can embed the petal $u_0u^2_1u^2_2\dotsb u^2_{c_2-1}u_0$ as the subgraph composed of $D_1$, $\frac{c_2-c_2'}{k}-1$ elements of $\{D_2,D_3,\dotsc,D_{\frac{k-2}{2}}\}\setminus\{D_z\}$, and $c_2'$ loops in $D_0$. Next, we can embed the petal $u_0u^3_1u^3_2\dotsb u^3_{c_3-1}u_0$ as the subgraph composed of $D_z$, $\frac{c_3-c_3'}{k}-1$ elements of $\{D_2,D_3,\dotsc,D_{\frac{k-2}{2}}\}\setminus\{D_z\}$, and $c_3'$ loops in $D_0$.

Case $3.2$: $c_2'+c_3'\geq k$. Since $\frac{k^2}{2}-\frac{k}{2}\leq n\leq\frac{k^2}{2}-1$ and $n\mod{c_2'+c_3'}{k}$, we have $\frac{k}{2}\leq c_2'+c_3'-k\leq k-1$, which implies $\frac{3k}{2}\leq c_2'+c_3'\leq 2k-1$. Hence, $\min\{c_2',c_3'\}\geq\frac{3k}{2}-\max\{c_2',c_3'\}\geq\frac{3k}{2}-(k-1)>\frac{k}{2}$. Let $H_1'$ be the subgraph that contains $c_2'-\frac{k}{2}$ loops in $D_0$ together with the cycle $v_0v_2v_4\dotsb v_{k-2}v_0$ in $D_2$, and let $H_2'$ be the subgraph that contains $c_3'-\frac{k}{2}$ loops in $D_0$ together with the cycle $v_1v_3v_5\dotsb v_{k-1}v_1$ in $D_2$. Let $H_1$ be the subgraph that contains $D_1$, $\frac{c_2-c_2'}{k}-1$ elements of $\{D_3,\dotsc,D_{\frac{k-2}{2}}\}\setminus\{D_z\}$, and $H_1'$; let $H_2$ be the subgraph that contains $D_z$, $\frac{c_3-c_3'}{k}-1$ elements of $\{D_3,\dotsc,D_{\frac{k-2}{2}}\}\setminus\{D_z\}$, and $H_2'$. It is easy to see that we can embed the petals $u_0u^2_1u^2_2\dotsb u^2_{c_2-1}u_0$ as $H_1$ and $u_0u^3_1u^3_2\dotsb u^3_{c_3-1}u_0$ as $H_2$.
\end{proof}

\subsection{Cycles with one chord}

Unlike the case when $k$ is odd, the obstruction of the triangle in $C_n^{\{0,2\}}$ when embedded in $K_k^*$ only occurs when $k=4$. This is because when $k$ is even, the chord $w_0w_j$ in $C_n^{\{0,j\}}$ can be embedded as a diagonal $v_0v_{\frac{k}{2}}$, instead of as a loop in $K_k^*$ when $k$ is odd. Hence, with the exception of $(k,n,j)=(4,8,2)$, the trivial necessary conditions are also sufficient for embedding $C_n^{\{0,j\}}$ in $K_k^*$.

\begin{theorem}\label{chordedembedevenk}
Let $k$ be even. Then $C_n^{\{0,j\}}$ can be embedded in $K_k^*$ if and only if $k\geq4$, $n\leq\frac{k^2}{2}$, and $(k,n,j)\neq(4,8,2)$. 
\end{theorem}

\begin{proof}
For the only if direction, if $k\leq2$, then $n\leq\frac{2^2}{2}=2$, which is impossible according to Definition \ref{chordedcycle}. Hence, $k\geq4$. Since the chorded cycle $C_n^{\{0,j\}}$ has only two vertices with odd degrees, when it is embedded in $K_k^*$, the image will have at most two vertices with odd degrees. Since all the $k$ vertices in $K_k^*$ have odd degrees, after embedding, at least one edge will be missing per vertex from at least $k-2$ vertices. In other words, there are at least $\frac{k-2}{2}$ edges missing, implying that the number of edges in $C_n^{\{0,j\}}$ is at most $\binom{k+1}{2}-\frac{k-2}{2}=\frac{k^2+2}{2}$. Since $C_n^{\{0,j\}}$ has $n+1$ edges, we have $n\leq\frac{k^2}{2}$.

If $(k,n,j)=(4,8,2)$, when $C_8^{\{0,2\}}$ is embedded in $K_4^*$, at least one edge that connects two vertices will be missing as discussed. Without loss of generality, let this edge be $v_1v_3$. Note that the subgraph given by the cycle $w_0w_1w_2w_0$ in $C_8^{\{0,2\}}$ must be embedded as a triangle in $K_4^*-v_1v_3$. Without loss of generality, let $w_0w_1w_2w_0$ be embedded as $v_0v_1v_2v_0$. The remaining edges of $K_4^*$, namely $K_4^*-v_1v_3-v_0v_1v_2v_0$, form a disconnected graph. As a result, it is impossible to embed the subgraph $w_2w_3\dotsb w_7w_0$ of $C_8^{\{0,2\}}$. Therefore, if $C_n^{\{0,j\}}$ can be embedded in $K_k^*$ , then $(k,n,j)\neq(4,8,2)$.

For the if direction, we are going to show that if $k\geq4$, $n\leq\frac{k^2}{2}$, and $(k,n,j)\neq(4,8,2)$, then $C_n^{\{0,j\}}$ can be embedded in $(K_k^*-I)\cup\{v_0v_{\frac{k}{2}}\}$, which contains $\frac{k^2}{2}+1$ edges. In all of the following cases, we embed the chord $w_0w_j$ as $v_0v_{\frac{k}{2}}$ in $K_k^*$.

Without loss of generality, we can assume that $n>\binom{k}{2}-3$. Otherwise, if $4\leq n\leq\binom{k}{2}-3$, we can embed $C_n^{\{0,j\}}$ in $K_{k-1}^*$ by Theorem \ref{OddChordedCycle}, which can be further embedded in $K_k^*$ trivially. Let $h=\frac{k^2}{2}-n$ be the gap between $n$ and $\frac{k^2}{2}$. Based on our assumption, $0\leq h\leq\frac{k^2}{2}-\left(\binom{k}{2}-2\right)=\frac{k}{2}+2$.

Case $1$: $j\leq\frac{k}{2}$. Let $H_0$, $H_1$, and $H_2$ be subgraphs of $K_k^*-I$ such that $H_0=\{v_1v_1,v_2v_2,\dotsc,\linebreak v_hv_h\}$, $H_1=v_0v_1v_2\dotsb v_{j-1}v_{\frac{k}{2}}$, and $H_2=K_k^*-I-H_0-H_1$. Note that $H_0$ is well-defined since $h\leq\frac{k}{2}+2\leq\frac{k}{2}+\frac{k}{2}=k$, and $v_kv_k$ represents the loop $v_0v_0$. With a simple count, we see that $H_0$ has $h$ edges, $H_1$ has $j$ edges, and $H_2$ has $\binom{k+1}{2}-\frac{k}{2}-h-j=\frac{k^2}{2}-h-j=n-j$ edges. It is clear that the only two odd degree vertices in $H_1$ and $H_2$ are $v_0$ and $v_{\frac{k}{2}}$. Furthermore, we claim that $H_2$ is connected: if $k=4$, then $(k,n,j)\neq(4,8,2)$ implies that $h\neq\frac{4^2}{2}-8=0$, so $H_2$ contains the path $v_0v_3v_2$ but not the loop $v_1v_1$; if $k\geq6$, then $H_2$ contains $(D_2\cup\{v_{k-1}v_0\})-\{v_{j-1}v_{\frac{k}{2}}\}$, so $H_2$ is connected.  Therefore, $H_1$ and $H_2$ both contain an Eulerian trail, and we can embed the paths $w_0w_1w_2\dotsb w_j$ and $w_0w_{n-1}w_{n-2}\dotsb w_j$ as $H_1$ and $H_2$ respectively.

Case $2$: $j>\frac{k}{2}$. Let $j'$ and $j''$ be the smallest positive integers such that $j-\frac{k}{2}\mod{j'}{k}$ and $j'\mod{j''}{\frac{k}{2}}$. Let
$$H_0=\begin{cases}
\{v_{j''+1}v_{j''+1},v_{j''+2}v_{j''+2},\dotsc,v_{j''+h}v_{j''+h}\}&\text{if }h\leq\frac{k}{2};\\
\{v_{j''+1}v_{j''+1},v_{j''+2}v_{j''+2},\dotsc,v_{j''+h-\frac{k}{2}}v_{j''+h-\frac{k}{2}}\}\\
\hspace{30pt}\cup\{v_0v_2v_4\dotsb v_{k-2}v_0\}&\text{if }h>\frac{k}{2},\\
\end{cases}$$
$H_1$ contain the path $v_0v_1v_2\dotsb v_\frac{k}{2}$ along with $\frac{j-\frac{k}{2}-j'}{k}$ elements of $\{D_3,D_4,\dotsc,D_{\frac{k-2}{2}}\}$ and the edges
$$\begin{cases}
\{v_1v_1,v_2v_2,\dotsc,v_{j'}v_{j'}\}&\text{if }j'\leq\frac{k}{2};\text{ and}\\
\{v_1v_1,v_2v_2,\dotsc,v_{j''}v_{j''}\}\cup\{v_1v_3v_5\dotsb v_{k-1}v_1\}&\text{if }j'>\frac{k}{2},
\end{cases}$$
and $H_2=K_k^*-I-H_0-H_1$. Note that $v_kv_k$ represents the loop $v_0v_0$. With a simple count, we see that $H_0$ has $h$ edges, $H_1$ has $j$ edges, and $H_2$ has $\binom{k+1}{2}-\frac{k}{2}-h-j=\frac{k^2}{2}-h-j=n-j$ edges. It is clear that the only two odd degree vertices in $H_1$ and $H_2$ are $v_0$ and $v_{\frac{k}{2}}$. It is also clear that $H_1$ is connected. We will now show that $H_2$ is connected.

\begin{itemize}
\item If $k=4$, then $H_2$ is connected since it contains the path $v_0v_3v_2$ and does not contain the loop $v_1v_1$.
\item If $k=6$, then $H_2$ contains the path $v_0v_5v_4v_3$. Hence, the only possible disconnected edges in $H_2$ are $v_1v_1$, $v_2v_2$, and $v_1v_2$. However, $v_1v_1$ and $v_1v_2$ are always contained in $H_1$. Moreover, $H_2$ contains $v_2v_2$ if and only if $(j',h)\in\{(1,0),(4,0)\}$, and in both cases, $H_2$ also contains the cycle $v_0v_2v_4v_0$, so $H_2$ is connected.
\item If $k=8$, then $H_2$ contains the path $v_0v_7v_6v_5v_4$. Hence, the only possible disconnected edges in $H_2$ are $v_1v_1$, $v_2v_2$, $v_3v_3$, $v_1v_2$, $v_2v_3$, and $v_1v_3$.  However, $v_1v_1$, $v_1v_2$, and $v_2v_3$ are always contained in $H_1$. The loop $v_2v_2$ is contained in $H_2$ if and only if $(j',h)\in\{(1,0),(5,0)\}$, and in both cases, $H_2$ also contains the cycle $v_0v_2v_4v_6v_0$. Finally, note that if $j'\leq\frac{k}{2}$, then $H_2$ contains the cycle $v_1v_3v_5v_7v_1$; if $j'>\frac{k}{2}=4$, then $\frac{j-\frac{k}{2}-j'}{k}<\frac{16-4-4}{8}\leq1$, so $H_2$ contains $D_3$. Hence, vertex $v_3$ is always connected to $v_0$ in $H_2$, thus the edges $v_1v_3$ and $v_3v_3$ will never be isolated edges in $H_2$. Therefore, $H_2$ is always connected.
\item If $k\geq10$, then $H_2$ contains $\{v_0v_{k-1}v_{k-2}\dotsb v_{\frac{k}{2}}\}\cup D_i$ for some $i\in\{3,4,\dotsc,\frac{k-2}{2}\}$. To demonstrate this, let us first assume the contrary, i.e., every $D_i$, $i\in\{3,4,\dotsc,\frac{k-2}{2}\}$, is in $H_1$. Hence, $\frac{j-\frac{k}{2}-j'}{k}=\frac{k-2}{2}-2$, which implies $j-j'=\frac{k^2-5k}{2}$. Since $j\leq\frac{n}{2}\leq\frac{1}{2}\cdot\frac{k^2}{2}=\frac{k^2}{4}$ and $j'>0$, we have $\frac{k^2-5k}{2}<\frac{k^2}{4}$, which occurs if and only if $0<k<10$, contradicting that $k\geq10$. Therefore, $H_2$ is connected.
\end{itemize}
Therefore, $H_1$ and $H_2$ both contain an Eulerian trail, and we can embed the paths \linebreak $w_0w_1w_2\dotsb w_j$ and $w_0w_{n-1}w_{n-2}\dotsb w_j$ as $H_1$ and $H_2$ respectively.
\end{proof}

\subsection{Spider graphs with four legs}

\begin{theorem}\label{spiderembedevenk}
Let $k$ be even, and let $n=\ell_1+\ell_2+\ell_3+\ell_4$. Then $S_{\ell_1,\ell_2,\ell_3,\ell_4}$ can be embedded in $K_k^*$ if and only if $k\geq4$, $n\leq\frac{k^2+4}{2}$, and $(k,\ell_1)\neq(4,2)$.
\end{theorem}

\begin{proof}
For the only if direction, if $k\leq2$, then the number of edges in $K_k^*$ is at most $\binom{2+1}{2}=3$, which is less than the number of edges in $S_{\ell_1,\ell_2,\ell_3,\ell_4}$, so $k\geq4$ is a necessary condition.

Since the spider graph $S_{\ell_1,\ell_2,\ell_3,\ell_4}$ has only four vertices with odd degrees, when it is embedded in $K_k^*$, the image will have at most four vertices with odd degrees. Since all the $k$ vertices in $K_k^*$ have odd degrees, after embedding, at least one edge will be missing per vertex from at least $k-4$ vertices. In other words, there are at least $\frac{k-4}{2}$ edges missing, implying that the number of edges in $S_{\ell_1,\ell_2,\ell_3,\ell_4}$ is at most $\binom{k+1}{2}-\frac{k-4}{2}=\frac{k^2+4}{2}$, i.e., $n\leq\frac{k^2+4}{2}$.

Finally, embedding $S_{\ell_1,\ell_2,\ell_3,\ell_4}$ in $K_k^*$ when $(k,\ell_1)=(4,2)$ is impossible due to Proposition \ref{Deltanbdeg>1}, since $\Delta(S_{\ell_1,\ell_2,\ell_3,\ell_4})=4$. Therefore, $(k,\ell_1)\neq(4,2)$.

For the if direction, by Proposition \ref{subgraph}, we only need to consider $n=\frac{k^2+4}{2}$. We are going to define the embedding $\phi:V(S_{\ell_1,\ell_2,\ell_3,\ell_4})\to V(K_4^*)$ explicitly.

When $k=4$, since $n=10$ and $\ell_1\neq2$,
$$(\ell_1,\ell_2,\ell_3,\ell_4)\in\{(1,1,1,7),(1,1,2,6),(1,1,3,5),(1,1,4,4),(1,2,2,5),(1,2,3,4),(1,3,3,3)\}.$$
For all these instances, we embed the first leg $x_0x^1_1$ of $S_{\ell_1,\ell_2,\ell_3,\ell_4}$ in $K_4^*$ as the loop $v_0v_0$, i.e., $\phi(x_0)=v_0$ and $\phi(x^1_1)=v_0$. The embedding of other legs are given by the following table.
\begin{center}
\begin{tabu}{|c|[1.5pt]c|c|c|}
\hline
$(\ell_1,\ell_2,\ell_3,\ell_4)$& $\big(\phi(x^2_1),\dotsc,\phi(x^2_{\ell_2})\big)$& $\big(\phi(x^3_1),\dotsc,\phi(x^3_{\ell_3})\big)$& $\big(\phi(x^4_1),\dotsc,\phi(x^4_{\ell_4})\big)$\\
\tabucline[1.5pt]{-}
$(1,1,1,7)$& $(v_1)$& $(v_2)$& $(v_3,v_3,v_2,v_2,v_1,v_1,v_3)$\\
\hline
$(1,1,2,6)$& $(v_1)$& $(v_2,v_2)$& $(v_3,v_3,v_2,v_1,v_1,v_3)$\\
\hline
$(1,1,3,5)$& $(v_1)$& $(v_2,v_2,v_3)$& $(v_3,v_3,v_1,v_1,v_2)$\\
\hline
$(1,1,4,4)$& $(v_1)$& $(v_2,v_2,v_3,v_3)$& $(v_3,v_1,v_1,v_2)$\\
\hline
$(1,2,2,5)$& $(v_1,v_1)$& $(v_2,v_2)$& $(v_3,v_3,v_1,v_2,v_3)$\\
\hline
$(1,2,3,4)$& $(v_1,v_1)$& $(v_2,v_2,v_3)$& $(v_3,v_3,v_1,v_2)$\\
\hline
$(1,3,3,3)$& $(v_1,v_1,v_2)$& $(v_2,v_2,v_3)$& $(v_3,v_3,v_1)$\\
\hline
\end{tabu}
\end{center}

When $k\geq6$, we proceed by considering the following cases.

Case $1$: $\ell_2+\ell_3\geq4$. Let $(i,j)=(2,3)$ if $\ell_2+\ell_3-1\leq\ell_1+\ell_4-2$, and let $(i,j)=(3,2)$ otherwise. Furthermore, let $c_i=\ell_2+\ell_3-1$ and $c_j=\ell_1+\ell_4-2$. We claim that the petal graph $P_{1,c_2,c_3}$ can be embedded in $K_k^*-I$. Since $\ell_2+\ell_3\geq4$, we have $\ell_2+\ell_3-1\geq3$. Also, since $\ell_4\geq\frac{k^2+4}{8}$ and $k\geq6$, we have $\ell_4\geq5$, so $\ell_1+\ell_4-2\geq4>3$. Finally, $\ell_1+\ell_2+\ell_3+\ell_4=\frac{k^2+4}{2}$, which implies $1+c_2+c_3=1+(\ell_2+\ell_3-1)+(\ell_3+\ell_4-2)=\frac{k^2}{2}$. Therefore, there exists an embedding $\psi:V(P_{1,c_2,c_3})\to V(K_k^*-I)$ by the argument provided in the proof of Theorem \ref{petalembedevenk}.

Now, we start embedding $S_{\ell_1,\ell_2,\ell_3,\ell_4}$ in $K_4^*$ as follows. We embed the second leg \linebreak $x_0x^2_1x^2_2\dotsb x^2_{\ell_2}$ of $S_{\ell_1,\ell_2,\ell_3,\ell_4}$ as the path $\psi(u_0)\psi(u^i_1)\psi(u^i_2)\dotsb\psi(u^i_{\ell_2})$ in $K_k^*$. Let $v_a=\psi(u_{\ell_2}^i)$ be a vertex in $K_k^*$. Then we embed the third leg $x_0x^3_1x^3_2\dotsb x^3_{\ell_3}$ as the path $\psi(u_0)\psi(u^i_{c_i-1})\psi(u^i_{c_i-2})\dotsb\linebreak\psi(u^i_{\ell_2})v_{\frac{k}{2}+a}$, where $\frac{k}{2}+a$ is performed under modulo $k$. Note that the edge $\psi(u^i_{\ell_2})v_{\frac{k}{2}+a}$ is in the perfect matching $I$.

Let $v_{b_2}=\psi(u^j_{\ell_4-2})$ and $v_{b_1}=\psi(u^j_{\ell_4-1})$ be vertices in $K_k^*$. Here, if $\ell_4-1=c_j$, then $u^j_{\ell_4-1}=u_0$.

Case $1.1$: $v_{b_2}\notin\{v_a,v_{\frac{k}{2}+a}\}$. In this case, we embed the fourth leg $x_0x^4_1x^4_2\dotsb x^4_{\ell_4}$ as the path $\psi(u_0)\psi(u_0)\psi(u^j_1)\psi(u^j_2)\dotsb\psi(u^j_{\ell_4-2})v_{\frac{k}{2}+b_2}$, where $\frac{k}{2}+b_2$ is performed under modulo $k$. Note that the edge $\psi(u^j_{\ell_4-2})v_{\frac{k}{2}+b_2}$ is in the perfect matching $I$ and is distinct from the edge $\psi(u^i_{\ell_2})v_{\frac{k}{2}+a}$. Finally, we embed the first leg $x_0x^1_1x^1_2\dotsb x^1_{\ell_1}$ as the path $\psi(u_0)\psi(u^j_{c_j-1})\psi(u^j_{c_j-2})\dotsb\linebreak\psi(u^j_{\ell_4-2})$.

Case $1.2$: $v_{b_2}\in\{v_a,v_{\frac{k}{2}+a}\}$ and $v_{b_1}\notin\{v_a,v_{\frac{k}{2}+a}\}$. In this case, we embed the fourth leg $x_0x^4_1x^4_2\dotsb x^4_{\ell_4}$ as the path $\psi(u_0)\psi(u_0)\psi(u^j_1)\psi(u^j_2)\dotsb\psi(u^j_{\ell_4-1})$, and we embed the first leg $x_0x^1_1x^1_2\dotsb x^1_{\ell_1}$ as the path $\psi(u_0)\psi(u^j_{c_j-1})\psi(u^j_{c_j-2})\dotsb\psi(u^j_{\ell_4-1})v_{\frac{k}{2}+b_1}$, where $\frac{k}{2}+b_1$ is performed under modulo $k$. Note that the edge $\psi(u^j_{\ell_4-1})v_{\frac{k}{2}+b_1}$ is in the perfect matching $I$ and is distinct from the edge $\psi(u^i_{\ell_2})v_{\frac{k}{2}+a}$.

Case $1.3$: $v_{b_2},v_{b_1}\in\{v_a,v_{\frac{k}{2}+a}\}$. In this case, we claim that $\ell_1>1$. To demonstrate this, let us first assume the contrary, i.e., $\ell_1=1$. As a result, $c_j=\ell_4-1$ and $v_{b_1}=\psi(u^j_{c_j})=\psi(u_0)$, which is the vertex $v_0$ due to the proof of Theorem \ref{petalembedevenk}. Therefore, we have $v_{b_2},v_{b_1}\in\{v_0,v_{\frac{k}{2}}\}$. However, $v_{b_2}=\psi(u^j_{c_j-1})$ cannot be the vertex $v_0$, or else $\psi(u_0)\psi(u^j_{c_j-1})=v_0v_0=\psi(u_0)\psi(u_0)$, contradicting that $\psi$ is an embedding. The vertex $v_{b_2}$ cannot be the vertex $v_{\frac{k}{2}}$ either, or else $\psi(u_0)\psi(u^j_{c_j-1})=v_0v_{\frac{k}{2}}$ is an edge in the perfect matching $I$, contradicting that $\psi$ is an embedding of $P_{1,c_2,c_3}$ in $K_k^*-I$. Hence, $\ell_1>1$.

Note that $v_{b_2}=v_{b_1}$, or else $\psi(u^j_{\ell_4-2})\psi(u^j_{\ell_4-1})=v_{b_2}v_{b_1}$ is an edge in the perfect matching $I$. Let $v_{b_0}=\psi(u^j_{\ell_4})$. Here, if $\ell_4=c_j$, then $u^j_{\ell_4}=u_0$. Since $v_{b_1}v_{b_0}=\psi(u^j_{\ell_4-1})\psi(u^j_{\ell_4})\neq\psi(u^j_{\ell_4-2})\psi(u^j_{\ell_4-1})=v_{b_2}v_{b_1}$ and $v_{b_1}v_{b_0}=\psi(u^j_{\ell_4-1})\psi(u^j_{\ell_4})$ is not in $I$, we conclude that $v_{b_0}\notin\{v_a,v_{\frac{k}{2}+a}\}$. As a result, we can embed the fourth leg $x_0x^4_1x^4_2\dotsb x^4_{\ell_4}$ as the path \linebreak $\psi(u_0)\psi(u^j_1)\psi(u^j_2)\dotsb\psi(u^j_{\ell_4})$, and we embed the first leg $x_0x^1_1x^1_2\dotsb x^1_{\ell_1}$ as the path \linebreak $\psi(u_0)\psi(u_0)\psi(u^j_{c_j-1})\psi(u^j_{c_j-2})\dotsb\psi(u^j_{\ell_4})v_{\frac{k}{2}+b_0}$, where $\frac{k}{2}+b_0$ is performed under modulo $k$. Note that the edge $\psi(u^j_{\ell_4})v_{\frac{k}{2}+b_0}$ is in the perfect matching $I$ and is distinct from the edge $\psi(u^i_{\ell_2})v_{\frac{k}{2}+a}$.

Case $2$: $\ell_2+\ell_3<4$. This means that $(\ell_2,\ell_3)\in\{(1,1),(1,2)\}$, and consequently $(\ell_1,\ell_2,\ell_3,\ell_4)\in\{(1,1,1,n-3),(1,1,2,n-4)\}$. In both cases, we are going to embed $S_{\ell_1,\ell_2,\ell_3,\ell_4}$ in $K_k^*$ such that the image of the embedding is $(K_k^*-I)\cup\{v_0v_{\frac{k}{2}},v_1v_{\frac{k}{2}+1}\}$: the first leg $x_0x^1_1$ is embedded as the loop $v_0v_0$, and the second leg $x_0x^2_1$ is embedded as the diagonal $v_0v_{\frac{k}{2}}$. If $\ell_3=1$, then the third leg $x_0x^3_1$ is embedded as the edge $v_0v_1$; if $\ell_3=2$, then the third leg $x_0x^3_1x^3_2$ is embedded as the path $v_0v_1v_1$. Note that $v_0$, $v_1$, $v_{\frac{k}{2}}$, and $v_{\frac{k}{2}+1}$ are all the odd degree vertices in $(K_k^*-I)\cup\{v_0v_{\frac{k}{2}},v_1v_{\frac{k}{2}+1}\}$, while $v_1$ and $v_{\frac{k}{2}}$ are all the odd degree vertices in the image graph of the first three legs. Hence, if we remove the image graph of the first three legs from $(K_k^*-I)\cup\{v_0v_{\frac{k}{2}},v_1v_{\frac{k}{2}+1}\}$, we have an Eulerian graph with $\ell_4$ edges and exactly two odd degree vertices, namely $v_0$ and $v_{\frac{k}{2}+1}$. Therefore, we can embed the fourth leg $x_0x^4_1x^4_2\dotsb x^4_{\ell_4}$ as the remaining graph of $(K_k^*-I)\cup\{v_0v_{\frac{k}{2}},v_1v_{\frac{k}{2}+1}\}$.
\end{proof}

\section{Edge-distinguishing chromatic numbers}\label{sec:edcn}

\subsection{Petal graphs}

\begin{theorem}\label{petalEDCN}
Let $c_2,c_3\in\N$ such that $3\leq c_2\leq c_3$. Let $e=1+c_2+c_3$ be the number of edges of $P_{1,c_2,c_3}$. The edge-distinguishing chromatic number of $P_{1,c_2,c_3}$ is given by
$$\lambda(P_{1,c_2,c_3})=\begin{cases}
5& \text{if $e\leq10$;}\\
\left\lceil\frac{-1+\sqrt{8e+1}}{2}\right\rceil& \text{if $e\geq11$ and $\left\lceil\frac{-1+\sqrt{8e+1}}{2}\right\rceil$ is odd; and}\\
\left\lceil\sqrt{2e}\right\rceil& \text{otherwise.}
\end{cases}$$
\end{theorem}

\begin{proof}
Since the maximum degree of $P_{1,c_2,c_3}$ is $6$, we must have $k\geq5$ in order to embed $P_{1,c_2,c_3}$ in $K_k^*$. On the other hand, by Theorem \ref{petalembedoddk}, $P_{1,c_2,c_3}$ can be embedded in $K_5^*$ if $e\leq10$. Hence, $\lambda(P_{1,c_2,c_3})=5$ if $e\leq10$.

If $e\geq11$, note that
$$5\leq\left\lceil\frac{-1+\sqrt{8e+1}}{2}\right\rceil\leq\left\lceil\sqrt{2e}\right\rceil\leq\left\lceil\frac{-1+\sqrt{8e+1}}{2}\right\rceil+1.$$
Also note that the minimum positive integer $k$ that satisfies $\binom{k+1}{2}\geq e$ is $\left\lceil\frac{-1+\sqrt{8e+1}}{2}\right\rceil$, and the minimum positive integer $k$ that satisfies $\frac{k^2}{2}\geq e$ is $\left\lceil\sqrt{2e}\right\rceil$.

If $\left\lceil\frac{-1+\sqrt{8e+1}}{2}\right\rceil$ is odd, then $\left\lceil\frac{-1+\sqrt{8e+1}}{2}\right\rceil-1$ is even but strictly less than $\left\lceil\sqrt{2e}\right\rceil$, so $\left\lceil\frac{-1+\sqrt{8e+1}}{2}\right\rceil-1$ does not satisfy $\frac{k^2}{2}\geq e$. Hence, by Theorems \ref{petalembedoddk} and \ref{petalembedevenk}, $k=\left\lceil\frac{-1+\sqrt{8e+1}}{2}\right\rceil$ is the minimum positive integer such that $P_{1,c_2,c_3}$ can be embedded in $K_k^*$.

If $\left\lceil\frac{-1+\sqrt{8e+1}}{2}\right\rceil$ and $\left\lceil\sqrt{2e}\right\rceil$ are both even, then $\left\lceil\sqrt{2e}\right\rceil=\left\lceil\frac{-1+\sqrt{8e+1}}{2}\right\rceil$, so $\left\lceil\sqrt{2e}\right\rceil-1$ does not satisfy $\binom{k+1}{2}\geq e$. Hence, by Theorems \ref{petalembedoddk} and \ref{petalembedevenk}, $k=\left\lceil\sqrt{2e}\right\rceil$ is the minimum positive integer such that $P_{1,c_2,c_3}$ can be embedded in $K_k^*$.

If $\left\lceil\frac{-1+\sqrt{8e+1}}{2}\right\rceil$ is even but $\left\lceil\sqrt{2e}\right\rceil$ is odd, then $\left\lceil\sqrt{2e}\right\rceil=\left\lceil\frac{-1+\sqrt{8e+1}}{2}\right\rceil+1$ is odd and satisfies $\binom{k+1}{2}\geq e$, but $\left\lceil\frac{-1+\sqrt{8e+1}}{2}\right\rceil$ does not satisfy $\frac{k^2}{2}\geq e$. Hence, by Theorems \ref{petalembedoddk} and \ref{petalembedevenk}, $k=\left\lceil\sqrt{2e}\right\rceil$ is the minimum positive integer such that $P_{1,c_2,c_3}$ can be embedded in $K_k^*$.

Finally, we finish by applying Theorem \ref{embedding}.
\end{proof}

\subsection{Cycles with one chord}

\begin{theorem}\label{chordedcycleEDCN}
Let $j,n\in\N$ such that $2\leq j\leq\frac{n}{2}$. Let $e=n+1$ be the number of edges of $C_n^{\{0,j\}}$. The edge-distinguishing chromatic number of $C_n^{\{0,j\}}$ is given by
$$\lambda(C_n^{\{0,2\}})=\begin{cases}
4& \text{if $e\leq8$;}\\
5& \text{if $e=9$;}\\
\left\lceil\frac{-1+\sqrt{8e+17}}{2}\right\rceil& \text{if $e\geq10$ and $\left\lceil\frac{-1+\sqrt{8e+17}}{2}\right\rceil$ is odd; and}\\
\left\lceil\sqrt{2e-2}\right\rceil& \text{otherwise,}
\end{cases}$$
and if $j\geq3$, then
$$\lambda(C_n^{\{0,j\}})=\begin{cases}
4& \text{if $e\leq6$;}\\
\left\lceil\frac{-1+\sqrt{8e+1}}{2}\right\rceil& \text{if $e\geq7$ and $\left\lceil\frac{-1+\sqrt{8e+1}}{2}\right\rceil$ is odd; and}\\
\left\lceil\sqrt{2e-2}\right\rceil& \text{otherwise.}
\end{cases}$$
\end{theorem}

\begin{proof}
It is obvious that no chorded cycle can be embedded in $K_3^*$, so $\lambda(C_n^{\{0,j\}})\geq4$ for all $j,n\in\N$ such that $2\leq j\leq\frac{n}{2}$.

When $j=2$, by Theorem \ref{chordedembedevenk}, $C_n^{\{0,2\}}$ can be embedded in $K_4^*$ if and only if $e\leq8$. Hence, $\lambda(C_n^{\{0,2\}})=4$ if $e\leq8$. Also, by Theorem \ref{OddChordedCycle}, $C_n^{\{0,2\}}$ can be embedded in $K_5^*$ when $e=9$. Hence, $\lambda(C_n^{\{0,2\}})=5$ if $e=9$.

If $e\geq10$, note that
$$5\leq\left\lceil\frac{-1+\sqrt{8e+17}}{2}\right\rceil\leq\left\lceil\sqrt{2e-2}\right\rceil\leq\left\lceil\frac{-1+\sqrt{8e+17}}{2}\right\rceil+1.$$
Also note that the minimum positive integer $k$ that satisfies $\binom{k+1}{2}-2\geq e$ is $\left\lceil\frac{-1+\sqrt{8e+17}}{2}\right\rceil$, and the minimum positive integer $k$ that satisfies $\frac{k^2}{2}+1\geq e$ is $\left\lceil\sqrt{2e-2}\right\rceil$.

If $\left\lceil\frac{-1+\sqrt{8e+17}}{2}\right\rceil$ is odd, then $\left\lceil\frac{-1+\sqrt{8e+17}}{2}\right\rceil-1$ is even but strictly less than $\left\lceil\sqrt{2e-2}\right\rceil$, so $\left\lceil\frac{-1+\sqrt{8e+17}}{2}\right\rceil-1$ does not satisfy $\frac{k^2}{2}+1\geq e$. Hence, by Theorems \ref{OddChordedCycle} and \ref{chordedembedevenk}, $k=\left\lceil\frac{-1+\sqrt{8e+17}}{2}\right\rceil$ is the minimum positive integer such that $C_n^{\{0,2\}}$ can be embedded in $K_k^*$.

If $\left\lceil\frac{-1+\sqrt{8e+17}}{2}\right\rceil$ and $\left\lceil\sqrt{2e-2}\right\rceil$ are both even, then $\left\lceil\sqrt{2e-2}\right\rceil=\left\lceil\frac{-1+\sqrt{8e+17}}{2}\right\rceil$, so \linebreak $\left\lceil\sqrt{2e-2}\right\rceil-1$ does not satisfy $\binom{k+1}{2}-2\geq e$. Hence, by Theorems \ref{OddChordedCycle} and \ref{chordedembedevenk}, $k=\left\lceil\sqrt{2e-2}\right\rceil$ is the minimum positive integer such that $C_n^{\{0,2\}}$ can be embedded in $K_k^*$.

If $\left\lceil\frac{-1+\sqrt{8e+17}}{2}\right\rceil$ is even but $\left\lceil\sqrt{2e-2}\right\rceil$ is odd, then $\left\lceil\sqrt{2e-2}\right\rceil=\left\lceil\frac{-1+\sqrt{8e+17}}{2}\right\rceil+1$ is odd and satisfies $\binom{k+1}{2}-2\geq e$, but $\left\lceil\frac{-1+\sqrt{8e+17}}{2}\right\rceil$ does not satisfy $\frac{k^2}{2}+1\geq e$. Hence, by Theorems \ref{OddChordedCycle} and \ref{chordedembedevenk}, $k=\left\lceil\sqrt{2e-2}\right\rceil$ is the minimum positive integer such that $C_n^{\{0,2\}}$ can be embedded in $K_k^*$.

Our proof of the formula for $\lambda(C_n^{\{0,2\}})$ is finished by applying Theorem \ref{embedding}.

When $j\geq3$, by Theorem \ref{chordedembedevenk}, $C_n^{\{0,j\}}$ can be embedded in $K_4^*$ if $e\leq6$. Hence, $\lambda(C_n^{\{0,j\}})=4$ if $e\leq6$. If $e\geq6$, note that
$$4\leq\left\lceil\frac{-1+\sqrt{8e+1}}{2}\right\rceil\leq\left\lceil\sqrt{2e-2}\right\rceil\leq\left\lceil\frac{-1+\sqrt{8e+1}}{2}\right\rceil+1.$$
Also note that the minimum positive integer $k$ that satisfies $\binom{k+1}{2}\geq e$ is $\left\lceil\frac{-1+\sqrt{8e+1}}{2}\right\rceil$, and the minimum positive integer $k$ that satisfies $\frac{k^2}{2}+1\geq e$ is $\left\lceil\sqrt{2e-2}\right\rceil$. The rest of the proof is analogous to the case when $j=2$.
\end{proof}

\subsection{Spider graphs with four legs}

\begin{theorem}\label{spiderEDCN}
Let $\ell_1,\ell_2,\ell_3,\ell_4\in\N$ such that $\ell_1\leq\ell_2\leq\ell_3\leq\ell_4$. Let $e=\ell_1+\ell_2+\ell_3+\ell_4$ be the number of edges of $S_{\ell_1,\ell_2,\ell_3,\ell_4}$. The edge-distinguishing chromatic number of $S_{\ell_1,\ell_2,\ell_3,\ell_4}$ is given by
$$\lambda(S_{1,1,1,\ell_4})=\begin{cases}
4& \text{if $e\leq10$;}\\
\left\lceil\frac{-1+\sqrt{8e+9}}{2}\right\rceil& \text{if $e\geq11$ and $\left\lceil\frac{-1+\sqrt{8e+9}}{2}\right\rceil$ is odd; and}\\
\left\lceil\sqrt{2e-4}\right\rceil& \text{otherwise,}
\end{cases}$$
and if $\ell_3\geq2$, then
$$\lambda(S_{\ell_1,\ell_2,\ell_3,\ell_4})=\begin{cases}
4& \text{if $e\leq10$ and $\ell_1=1$;}\\
5& \text{if $e\leq10$ and $\ell_1=2$;}\\
\left\lceil\frac{-1+\sqrt{8e+1}}{2}\right\rceil& \text{if $e\geq11$ and $\left\lceil\frac{-1+\sqrt{8e+1}}{2}\right\rceil$ is odd; and}\\
\left\lceil\sqrt{2e-4}\right\rceil& \text{otherwise.}
\end{cases}$$
\end{theorem}

\begin{proof}
Since the maximum degree of $S_{\ell_1,\ell_2,\ell_3,\ell_4}$ is $4$, we must have $\lambda(S_{\ell_1,\ell_2,\ell_3,\ell_4})\geq4$ by Proposition \ref{Deltanbdeg>1}. In other words, $k\geq4$ is necessary to embed $S_{\ell_1,\ell_2,\ell_3,\ell_4}$ in $K_k^*$.

When $\ell_3=1$, by Theorem \ref{spiderembedevenk}, $S_{1,1,1,\ell_4}$ can be embedded in $K_4^*$ if and only if $e\leq10$. Hence, $\lambda(S_{1,1,1,\ell_4})=4$ if $e\leq10$.

If $e\geq11$, note that
$$5\leq\left\lceil\frac{-1+\sqrt{8e+9}}{2}\right\rceil\leq\left\lceil\sqrt{2e-4}\right\rceil\leq\left\lceil\frac{-1+\sqrt{8e+9}}{2}\right\rceil+1.$$
Also note that the minimum positive integer $k$ that satisfies $\binom{k+1}{2}-1\geq e$ is $\left\lceil\frac{-1+\sqrt{8e+9}}{2}\right\rceil$, and the minimum positive integer $k$ that satisfies $\frac{k^2}{2}+2\geq e$ is $\left\lceil\sqrt{2e-4}\right\rceil$.

If $\left\lceil\frac{-1+\sqrt{8e+9}}{2}\right\rceil$ is odd, then $\left\lceil\frac{-1+\sqrt{8e+9}}{2}\right\rceil-1$ is even but strictly less than $\left\lceil\sqrt{2e-4}\right\rceil$, so $\left\lceil\frac{-1+\sqrt{8e+9}}{2}\right\rceil-1$ does not satisfy $\frac{k^2}{2}+2\geq e$. Hence, by Theorems \ref{spiderembedoddk} and \ref{spiderembedevenk}, $k=\left\lceil\frac{-1+\sqrt{8e+9}}{2}\right\rceil$ is the minimum positive integer such that $S_{1,1,1,\ell_4}$ can be embedded in $K_k^*$.

If $\left\lceil\frac{-1+\sqrt{8e+9}}{2}\right\rceil$ and $\left\lceil\sqrt{2e-4}\right\rceil$ are both even, then $\left\lceil\sqrt{2e-4}\right\rceil=\left\lceil\frac{-1+\sqrt{8e+9}}{2}\right\rceil$, so \linebreak $\left\lceil\sqrt{2e-4}\right\rceil-1$ does not satisfy $\binom{k+1}{2}-1\geq e$. Hence, by Theorems \ref{spiderembedoddk} and \ref{spiderembedevenk}, $k=\left\lceil\sqrt{2e-4}\right\rceil$ is the minimum positive integer such that $S_{1,1,1,\ell_4}$ can be embedded in $K_k^*$.

If $\left\lceil\frac{-1+\sqrt{8e+9}}{2}\right\rceil$ is even but $\left\lceil\sqrt{2e-4}\right\rceil$ is odd, then $\left\lceil\sqrt{2e-4}\right\rceil=\left\lceil\frac{-1+\sqrt{8e+9}}{2}\right\rceil+1$ is odd and satisfies $\binom{k+1}{2}-1\geq e$, but $\left\lceil\frac{-1+\sqrt{8e+9}}{2}\right\rceil$ does not satisfy $\frac{k^2}{2}+2\geq e$. Hence, by Theorems \ref{spiderembedoddk} and \ref{spiderembedevenk}, $k=\left\lceil\sqrt{2e-4}\right\rceil$ is the minimum positive integer such that $S_{1,1,1,\ell_4}$ can be embedded in $K_k^*$.

Our proof of the formula for $\lambda(S_{1,1,1,\ell_4})$ is finished by applying Theorem \ref{embedding}.

When $\ell_3\geq2$, by Theorem \ref{spiderembedevenk}, $S_{\ell_1,\ell_2,\ell_3,\ell_4}$ can be embedded in $K_4^*$ if and if $e\leq10$ and $\ell_1=1$. Hence, $\lambda(S_{\ell_1,\ell_2,\ell_3,\ell_4})=4$ if $e\leq10$ and $\ell_1=1$. If $\ell_1=2$, then $e\geq8\geq7$, and by Theorem \ref{spiderembedoddk}, $S_{\ell_1,\ell_2,\ell_3,\ell_4}$ can be embedded in $K_5^*$ if $e\leq10$. Hence, $\lambda(S_{\ell_1,\ell_2,\ell_3,\ell_4})=5$ if $e\leq10$ and $\ell_1=2$.

If $e\geq11$, note that
$$5\leq\left\lceil\frac{-1+\sqrt{8e+1}}{2}\right\rceil\leq\left\lceil\sqrt{2e-4}\right\rceil\leq\left\lceil\frac{-1+\sqrt{8e+1}}{2}\right\rceil+1.$$
Also note that the minimum positive integer $k$ that satisfies $\binom{k+1}{2}\geq e$ is $\left\lceil\frac{-1+\sqrt{8e+1}}{2}\right\rceil$, and the minimum positive integer $k$ that satisfies $\frac{k^2}{2}+2\geq e$ is $\left\lceil\sqrt{2e-4}\right\rceil$. The rest of the proof is analogous to the case when $\ell_3=1$.
\end{proof}

\section{Concluding remarks and future work}

When the authors determined the EDCN of spider graphs with three legs \cite{fw}, the focus was only on one graph, making the approach restrictive. In this paper, we consider a variety of graphs; in particular, petal graphs serve as an intermediate step when chorded cycles and spider graphs are embedded in $K_k^*$, which makes the proof easier to navigate. Moreover, this new approach allows us to solve all suggested problems listed in the aforementioned paper except the conjecture concerning caterpillar trees.

Although the embedding of caterpillar trees in $K_k^*$ still eludes us when $k$ is even, we are able to show that the trivial necessary conditions for embedding certain caterpillar trees in $K_k^*$ are also sufficient when $k$ is odd. Hence, one major future work direction is to continue our study of caterpillar trees.

Here is a list of other potential future projects on EDCN.
\begin{enumerate}
\item General petal graphs with three or four petals;
\item Spider graphs with more than $4$ legs;
\item Ladder graphs.
\end{enumerate}

\end{document}